\documentclass{article}

\usepackage{graphicx}

\usepackage{amsmath}
\usepackage{amssymb}
\usepackage{amsthm}
\usepackage{stackrel}
\usepackage{epstopdf}
\usepackage{xcolor}

\usepackage{geometry}
\newtheorem{lemma}{Lemma}
\newtheorem{proposition}{Proposition}
\newtheorem{defn}{Definition}
\newtheorem{remark}{Remark}

\newtheorem{assumption}{Assumption}
\DeclareMathOperator*{\argmax}{arg\,max}
\newcommand{\R}{\mathbb{R}}

\begin{document}

\title{Optimal and Sub-optimal Feedback Controls \\ for Biogas Production}

\author{Antoine Haddon$^{1,2}$ \and H\'ector Ram\'irez$^1$ \and Alain Rapaport$^2$}

\date{$^1$ Mathematical Engineering Department and Center for Mathematical Modelling (CNRS UMI 2807), Universidad de Chile, Beauchef 851, Santiago, Chile (ahaddon@dim.uchile.cl) \\
$^2$ MISTEA, Universit\'e Montpellier, INRA, Montpellier SupAgro, 2 pl. Viala, 34060 Montpellier, France
}

\maketitle

\begin{abstract}

We revisit the optimal control problem of maximizing biogas production in continuous bio-processes in two directions:
1. over an infinite horizon,
2. with sub-optimal controllers independent of the time horizon.
For the first point, we identify a set of optimal controls for the problems with an averaged reward and with a discounted reward when the discount factor goes to 0 and we show that the value functions of both problems are equal.
For the finite horizon problem, our approach relies on a framing of the value function by considering a different reward for which the optimal solution has an explicit optimal feedback that is time-independent.
In particular, we show that this technique allows us to provide explicit bounds on the sub-optimality of the proposed controllers.
The various strategies are finally illustrated on Haldane and Contois growth functions.

\end{abstract}


\section{Introduction}

Anaerobic digestion  is  a  biological  process  in  which  organic matter is transformed by microbial species into biogas (such as methane and carbon dioxide).
Such transformations have been  used for a long time in waste water-treatment plants to purify water \cite{R06}.
Valorizing biogas production while treating wastewater has received recently great attention, as a way for producing valuable energy and limiting the carbon footprint of the process \cite{RM13}.
As a final product of the biological reaction, the total production of biogas measures the performances of the biological transformation.
Therefore, there is a strong interest in determining control strategies maximizing  biogas production.

With continuous-stirred bioreactors, two kinds of anaerobic models are usually considered for control purposes in the literature: the one-step model, which corresponds to the classical chemostat model \cite{HLRS17}, and the two-step model that has been proposed by Bernard et al. \cite{BHDGS01}.

Although these models only have few dynamic variables, it has been shown that they are capable of reproducing the qualitative behavior of the anaerobic digestion process \cite{Bmodelcomplex}.
Furthermore, in the two-step model, the second reaction is the most limiting due to inhibition by the substrate and we can then consider that a one-step model can be used to focus on the second reaction. In particular, a common assumption is to consider that the first step is fast and then the two reactions can be reduced to a single one with a slow-fast approximation and in this case, the one-step model provides a good representation of the biogas production.

The control variable is typically the input flow rate (or equivalently the dilution rate, since the volume of the reactor is constant in continuous operating mode).
Several works have already considered the static optimization problem of maximizing the output flow rate of biogas at steady state, and various control strategies have been proposed to stabilize the processes at these nominal states (see for instance \cite{SBRM99,RRMRL06,DK09,DK10,DK11,SMV12,VM15}).

There has been comparatively much less work considering the dynamic optimization problem over the transients, while bio-processes are often not initialized at their optimal nominal state.
Although the optimal control problem, which consists in maximizing  biogas production over a given time interval, has been posed a long time ago \cite{SLTPPS97}, it is still unsolved today (even for the one-step model).
Let us mention two attempts to solve approximately or partially this problem.
Sbarciog  et  al.  \cite{SLV12} have  considered the two-step anaerobic model and proposed  a strategy for maximizing biogas production as an optimal control to drive the system in finite time in a neighborhood of the optimal steady state, with additive penalty terms in the criterion.
In \cite{GSH15}, Ghouali et al. give a complete solution of the original optimal control problem for the one-step model, but for a particular subset of initial conditions which belong to an invariant manifold of the system (see also \cite{HLRS19}).
The dynamics can be then reduced to a scalar one and the authors show that the optimal solution exhibits a singular arc with a ``most rapid approach path'' optimal strategy.
Let us underline that optimal control problems over a fixed time horizon possess generally a time-dependent optimal synthesis, while the duration of process operation is often poorly known.
However, the scalar reduced problem exhibits the remarkable feature of having an optimal synthesis independent of the terminal time, which makes it quite attractive from an application view point.

The purpose of the present article is to propose new control strategies for the one-step model, as time-independent feedbacks for general initial conditions
\begin{itemize}
\item[-] either considering an infinite horizon,
\item[-] either considering sub-optimal controllers for the finite horizon.
\end{itemize}
For the infinite horizon (see for instance the book \cite{carlson2012infinite}), we consider the limit of the discounted criterion (when the discount factor tends to zero) and the average cost.
We study optimal strategies and compare their related optimal costs.
This study extends the preliminary results presented in the conference paper \cite{HRR18} and considers a large class of growth functions, that can be in particular density-dependent (such as the Contois law) or not  (such as the Monod or Haldane law).
Our work for the finite horizon exploits and extends an approximation technique presented in \cite{HHRR17}.
This consists, for a given initial condition, in framing the optimal solution by considering a different reward for which the optimal solution can be determined exactly and that possess the property of having a time-independent optimal synthesis (i.e. whatever is the time horizon, finite or infinite).
This technique has moreover the advantage of providing bounds on the sub-optimality of the controllers.
The results are again obtained for a large class of growth functions and we show that density dependent growth functions lead to more sophisticated feedback laws.

The paper is organized as follows. Section \ref{secprelim} specifies dynamics, control, criterion and hypotheses, and gives some preliminary results about controllability and asymptotic behavior of solutions.
Sections \ref{secinfinite} and \ref{secfinite} study the optimal solutions, respectively for the infinite and finite time horizons.
Finally, Section \ref{secappli} illustrates our results on various growth functions.

\section{Preliminaries}
\label{secprelim}

In this work, we consider the classical chemostat model \cite{HLRS17}.
This represents a well-mixed continuously fed bioreactor in which a substrate of concentration $s$ is treated (and then transformed into biogas) by a population of microorganisms of concentration $x$
\begin{align} \label{eq:chemosS}
  \dot{s} &= u(s_{in} - s) - \frac{1}{Y} \mu (s,x) x, \\
  \dot{x} &= \mu (s,x) x - u x.
  \label{eq:chemosX}
\end{align}
We denote $s_{in} >0$ the inflow concentration of substrate, $Y$ the yield coefficient, $\mu(\cdot,\cdot)$ the specific growth rate and $u$ the dilution rate, which is the control.

The biogas flowrate is assumed proportional to the growth rate so that the biogas produced during a time interval $[t_0,T]$ is proportional to
\[
 \int_{t_0}^T \mu(s(t),x(t)) x(t)\,dt
\]
and, without loss of generality, we will suppose that the proportionality coefficient as well as the yield coefficient are equal to 1.

We will consider the following class of growth functions :
\begin{assumption} \label{assum:growthfct}
 We suppose that $\mu : \mathbb{R}_{+}\times \mathbb{R}_{+} \rightarrow \mathbb{R}_+ $ is a Lipschitz continuous function that satisfies, for all $x>0$
\[
\mu(0,x)=0 \mbox{ and } \mu(s,x)>0 \mbox{ for } s>0 .
\]
We suppose as well that $x\mapsto \mu(s,x)$ is non increasing, which models crowding effects, and  $x \mapsto \mu(s,x)x$ is non decreasing, which models the fact that having more biomass provides at least the same growth.
\end{assumption}
A typical instance of this class is the Contois growth function, defined later in \eqref{contois}, but note that this class of functions also contains growth functions that depend only on the substrate concentration, such as the Monod \eqref{monod} and the Haldane \eqref{haldane} functions.

We will study the problem of maximizing the accumulated biogas for controls in the following set of admissible controls
\[
 {\cal U}(t_0,T) = \Big\{ u(\cdot) \in L^{\infty}(t_0,T ; \R) : u(t) \in [0,u_{\max}] \mbox{ for } t \in [t_0,T]  \Big\}
\]
with $t_0 \in \R$ and $T \in \R \cup \{+\infty\}$, and where $u_{max} >0$ is a given parameter that represents the maximal dilution rate.
We will consider initial conditions taken in the invariant set
$${\cal D} :=[0,s_{in}) \times (0,\infty)$$
which corresponds to the most common operating conditions.
Notice that for initial conditions in ${\cal D}$, any solution of \eqref{eq:chemosS}-\eqref{eq:chemosX} cannot reach $s = s_{in}$ in finite time and stays non negative.
Therefore the set ${\cal D}$ is (forward) invariant.

\subsection{Properties of the Dynamics}
On the invariant domain ${\cal D}$, we introduce the change of variables
\[
\zeta=(s,z) \quad \mbox{with} \quad z=\frac{x}{s_{in}-s} ,
\]
under which the dynamics become
\begin{align} \label{eq:chemosz}
\dot \zeta = \left[\begin{array}{c}\dot s\\ \dot z \end{array}\right]=f(\zeta,u) :=
 \left[\begin{array}{c}
  \Big(u - \mu \big(s,(s_{in} - s)z\big)z \Big) (s_{in} - s) \\
  \mu \big(s,(s_{in} - s)z\big) (1-z)z
 \end{array}\right].
\end{align}
We will denote $s_{t_0,\xi,u}(\cdot)$ and $ z_{t_0,\xi,u}(\cdot) $ the solution of \eqref{eq:chemosz}, with initial condition $\xi = (s_0,z_0)= (s(t_0),z(t_0) ) \in {\cal D}$ and control $u(\cdot) \in {\cal U}(t_0,T)$.
The cumulated biogas production becomes
\begin{equation}
 \int_{t_{0}}^T \phi \big( s_{t_0,\xi,u}(t),z_{t_0,\xi,u}(t) \big) z_{t_0,\xi,u}(t) \,dt
\end{equation}
with
\begin{equation} \label{eq:phi}
 \phi(s,z) = \mu \big(s, (s_{in}-s)z \big) (s_{in}-s)
\end{equation}
and we will denote
\begin{equation}
 \label{phibar}
 \overline \phi (z) = \max_{s \in (0,s_{in}) } \phi(s,z).
\end{equation}

We can now establish an important property of the controlled dynamics.
\begin{lemma} \label{lemcompinv}
 The trajectories of the system \eqref{eq:chemosz} for a given initial condition $\xi=(s_0,z_0) \in \cal D$, for all admissible controls, remain in the set
 \begin{equation} \label{Lxi}
 	{\cal L}(\xi)= [0,s_{in}] \times [\min(z_{0},1),\max(z_0,1)].
 \end{equation}
\end{lemma}

\begin{proof}
From Assumption \ref{assum:growthfct} we have that $\mu(\cdot,\cdot) \geqslant 0$ and since the solutions $z(\cdot)$ satisfy \eqref{eq:chemosz}, we then have the following
\[
\min(z_{0},1) \leqslant z_{t_0,\xi,u}(t) \leqslant \max(z_{0},1)
\]
for all $t \geqslant 0$, for any admissible control $u(\cdot)$.
\end{proof}

In the following, we consider initial conditions that guarantee the controllability of the $s$ variable.
\begin{assumption} \label{assum:umax}
 We suppose that the initial condition $\xi \in \cal D$ is such that
 \[
  \max_{(s,z) \in {\cal L}(\xi)} \mu \big( s,(s_{in}-s) z \big) z < u_{max}.
 \]
\end{assumption}
In practice, for a given initial condition it possible to choose $u_{\max}$ such that the previous inequality is satisfied.

We now define a class of feedbacks, that will play an important role,  and that are based on the notion of {\em most rapid approach path}, a well known concept in the theory of optimal control, see for example \cite{rapaport2005competition,hartl1987new}.

\begin{defn}
 For $(s,z) \in {\cal L}(\xi)$, we define the {\em most rapid approach feedback} to a given substrate level $s^* \in [0,s_{in})$, as
\begin{equation} \label{def:MRAPfeedback}
\psi_{s^*}(s,z)=\left|\begin{array}{ll}
0 & \mbox{if } s>s^*,\\
\mu(s^*,(s_{in}-s^*)z) \,z \quad &  \mbox{if } s=s^*,\\
u_{\max} & \mbox{if } s<s^*.
\end{array}\right.
\end{equation}
\end{defn}

Clearly, with Assumption \ref{assum:umax} this feedback is well defined, so that, associated with this control, for every initial condition $\xi \in \cal D$, there exists a unique absolutely continuous solution for the dynamics \eqref{eq:chemosz}.

\begin{lemma} \label{lem:sStarReachable}
 For any $\xi \in \cal D$ satisfying Assumption \ref{assum:umax}, a given substrate level $s^* \in (0,s_{in})$ is reachable in finite time with the feedback $\psi_{s^*}$.
\end{lemma}

\begin{proof}  First, using the monotonicity properties of $\mu(\cdot,\cdot)$ of Assumption \ref{assum:growthfct}, it is clear that $ \psi_{s^*}$ is admissible provided Assumption \ref{assum:umax} is satisfied.

To show that $s^*$ is reachable in finite time, it is enough to note that when $s_{t_0,\xi,\psi_{s^*}}(t) > s^*$, for $t$ in a given open interval $I$,
we have
\begin{align*}
 \dot s_{t_0,\xi,\psi_{s^*}} (t)  = - \mu \big(s,(s_{in} - s) z \big) z (s_{in} - s) \leqslant k_- < 0, \quad \forall \, t\in I
\end{align*}
with $ k_- = - \min_{s \in (s^*,s_{in}) } \mu \big(s,(s_{in} - s) \min(z_0,1) \big) \min(z_0,1) (s_{in} - s^*)$.
This insures that $s^*$ is always reachable in finite time from $s_0 > s^*$.

Analogously, if $s_{t_0,\xi,\psi_{s^*}}(t)<s^*$, for $t \in I$, we have from Assumption \ref{assum:umax}
\begin{align*}
 \dot s_{t_0,\xi,\psi_{s^*}} (t)  = \left[ u_{\max} - \mu \big(s,(s_{in} - s)z\big)z \right] (s_{in} - s) \geqslant k_+ > 0 , \quad \forall \, t\in I
\end{align*}
with $ k_+ = \left[ u_{\max} - \max_{s \in (0,s^*) } \mu \big(s,(s_{in} - s) \max(z_0,1) \big) \max(z_0,1) \right] (s_{in} - s^*)$.
Then $s^*$ is reachable from $s_0 < s^*$, again in finite time.

\end{proof}

\begin{remark}

It should be pointed out that there is a similarity with the {\em turnpike} property \cite{zaslavski2006turnpike,TZ15} when using  the controller \eqref{def:MRAPfeedback}.
The turnpike property has received great attention in the literature (see for instance \cite{grune2017relation,rapaport2005competition,hartl1987new,rapaport2007nonturnpike}), and recent results give sufficient optimality conditions \cite{FKJB17,FB17}.
However, we shall show in the next sections that the value $s^*$, which determines the turnpike, has to depend on the initial condition (excepted for the very particular case when the initial condition belongs to the invariant set $\{z=1\}$ that has been solved in \cite{GSH15}).
So, we are not in the usual framework of a single turnpike \cite{FKJB17,FB17} or isolated turnpikes \cite{RC04},  and the results of the literature do not apply.
\end{remark}

For the problem on an infinite horizon, we will consider \emph{persistently exciting} controls, which are defined as satisfying
\[
 \int_{t_0}^T u(t) \, dt \stackrel[T \rightarrow \infty]{}{\longrightarrow} \infty.
\]
 As the next Lemma shows, the trajectories associated with these controls are such that $z_{t_0,\xi,u}(t)$ converges to 1, which is essential in our approach. Furthermore, for non persistently exciting controls, $s_{t_0,\xi,u}(t)$ converges to 0 and thus the biogas production also converges to 0.
As a consequence, the controls that maximize biogas production are necessarily persistently exciting controls.

\begin{lemma} \label{limz}
 For all initial conditions $\xi \in \cal D$ and for all persistently exciting controls $u(\cdot) \in {\cal U}(0,\infty)$, we have
 \[
  \lim_{t \rightarrow \infty} z_{0,\xi,u}(t) = 1
 \]
 and
 \[
  \lim_{\delta \rightarrow 0} \int_{0}^\infty \delta e^{-\delta t}  z_{0,\xi,u}(t)  \,dt =\lim_{T \rightarrow \infty} \frac{1}{T} \int_{0}^T z_{0,\xi,u}(t) \,dt = 1.
 \]
 Moreover, for non persistently exciting controls, we have
 \[
 \lim_{t\to+\infty} s_{0,\xi,u}(t)=0 .
 \]
\end{lemma}

\begin{proof}

From equation \eqref{eq:chemosz}, the solution $z(\cdot)=z_{0,\xi,u}(\cdot)$ can be written as follows
\begin{equation}
    \label{eq:solz}
z(t)=\frac{z_0+e^{\int_{t_0}^t \mu(s(\tau),x(\tau))\,d\tau}}{1+z_0\Big(e^{\int_{t_0}^t \mu(s(\tau),x(\tau))\,d\tau}-1\Big)}
\end{equation}
where $s(\cdot)=s_{0,\xi,u}(\cdot)$, $x(\cdot)=x_{0,\xi,u}(\cdot)$.
From equation \eqref{eq:chemosX}, the solution $x(\cdot)$ is such that
\[
x(t)=x(t_0)e^{\int_{t_0}^t \big(\mu(s(\tau),x(\tau))-u(\tau)\big)\,d\tau} \; .
\]
Therefore, if the integral function
\begin{equation}
    \label{def:intmu}
t \mapsto \int_{t_0}^t \mu(s(\tau),x(\tau))\,d\tau , \; t \geq t_0
\end{equation}
is bounded, then $x(t)$ must converge asymptotically to $0$ when $t$ goes to $+\infty$ and $u(\cdot)$ is a persistently exciting control. Moreover, from equations \eqref{eq:chemosS}, \eqref{eq:chemosX} we have
\[
   \frac{d}{dt}\big(s(t)+x(t)\big)=u(t)\big(s_{in}-s(t)+x(t)\big)\\
\]
so that
\[
s(t)+x(t)=s_{in}+(s(t_0)+x(t_0)-s_{in}\big)e^{-\int_{t_0}^t u(\tau)\,d\tau}
\]
and then $s(t)$ must converge to $s_{in}$ when $t$ goes to $+\infty$.
Consequently, by continuity of the function $\mu$, there exists $T>t_0$ such that
\[
\mu(s(t),x(t)))>\mu(s_{in},0)/2>0
\]
for any $t>T$, which implies that the integral defined in \eqref{def:intmu} goes to $+\infty$ when $t$ goes to $+\infty$, which is a contradiction. We deduce that this integral cannot be bounded and from equation \eqref{eq:solz} that $z(t)$ converges to $1$ when $t$ goes to $+\infty$.

A proof of the equality of limits of the integrals
\[
  \lim_{\delta \rightarrow 0} \int_{0}^\infty \delta e^{-\delta t}  z_{0,\xi,u}(t)  \,dt =\lim_{T \rightarrow \infty} \frac{1}{T} \int_{0}^T z_{0,\xi,u}(t) \,dt
 \]
can be found in \cite[Lemma 3.5]{grune1998asymptotic}.
For the value of the limits we use the fact that $z_{0,\xi,u}(t)$ converges to 1 :
for all $\tilde \varepsilon > 0$, there exits a time $t_{\tilde \varepsilon}$ such that, for all $t\geqslant t_{\tilde \varepsilon}$,
\[
  |z_{0,\xi,u}(t)-1|< \tilde \varepsilon.
\]
Then, for all $T\geqslant \max( t_{\tilde \varepsilon} , t_{\tilde \varepsilon} / \tilde \varepsilon )$
\begin{align*}
 \left| \frac{1}{T} \int_{0}^T z_{0,\xi,u}(t) \,dt -1 \right|
 & \leqslant  \frac{1}{T}  \int_{0}^{t_{\tilde\varepsilon}}\left| z_{0,\xi,u}(t)-1 \right| \,dt
 +  \frac{1}{T} \int_{t_{\tilde\varepsilon}}^T \left| z_{0,\xi,u}(t)-1 \right| \,dt   \\
 &<  \frac{t_{\tilde \varepsilon}}{T}  |z_0-1|  + \left(1 - \frac{ t_{\tilde \varepsilon}}{T}\right) \tilde \varepsilon \\
 &<  \tilde \varepsilon \left( |z_0-1| +1 \right).
\end{align*}
With this, for all $\varepsilon >0$, we can take $\tilde \varepsilon = \varepsilon / ( |z_0-1| +1 ) $ and then we have, for $T\geqslant \max( t_{\tilde \varepsilon} , t_{\tilde \varepsilon} / \tilde \varepsilon )$
\[
\left| \frac{1}{T} \int_{0}^T z_{0,\xi,u}(t) \,dt -1 \right|  < \varepsilon.
\]

Finally, we prove that for non persistently exciting controls, $s_{0,\xi,u}(t)$ converges to 0. Therefore, suppose that $u(\cdot)$ is an admissible control with a finite integral and we define, for all $ t \geqslant 0$,
\[
I(t):= \int_{0}^t u(\tau) \,d\tau < \infty
\]
and
\[
\varphi (t) := (s_{in}-s_{0,\xi,u}(t)) e^{I(t)}.
\]
Then
\[
\varphi' (t) = \phi \big( s_{0,\xi,u}(t), z_{0,\xi,u}(t) \big) z_{0,\xi,u}(t) e^{I(t)} \geqslant 0
\]
and since $\varphi (t)$ is bounded, we can deduce that $\varphi (t)$ converges as $t$ goes to infinity.
Note as well that $\varphi'$ is absolutely continuous and thus uniformly continuous. We can therefore use Barbalat's Lemma \cite[Lemma 4.2]{Khalil} to get that $\varphi'(t)$ converges to 0. Then, as $z_{0,\xi,u}(t)$ cannot reach 0 (Lemma \ref{lemcompinv}), we have that $\phi \big( s_{0,\xi,u}(t), z_{0,\xi,u}(t) \big)$ must converge to 0 and by continuity we conclude that $s_{0,\xi,u}(t)$ converges to 0.

\end{proof}

%%%%%%%%%%%%%%%%%%%%%%%%%%%%%%%%%%%%%%%%%%%%%%%%%%%%%%%%%%%%%%%%%%%%%%%%%%%%%%%%%%%%%%

\section{Infinite Horizon and Average Reward}
\label{secinfinite}

In this section, we study the problem of maximizing biogas production over an infinite horizon.
Since the dynamics \eqref{eq:chemosz} are autonomous, without loss of generality, we can assume here that $t_0=0$ and we will then denote $s_{\xi,u}(\cdot)$ and $ z_{\xi,u}(\cdot)$ solutions of \eqref{eq:chemosz}.

We start by defining the average biogas production during a time interval $[0,T]$ as
\begin{align} \label{eq:avgReward}
 J^T(\xi,u(\cdot)) = \frac{1}{T} \int_0^T  \phi \big(s_{\xi,u}(t),z_{\xi,u}(t)\big) z_{\xi,u}(t) \,dt
\end{align}
and we consider the inferior and superior limits as $T$ goes to infinity
\begin{align}\label{averagerewardinftyinf}
 \underline J^{\infty}(\xi,u(\cdot)) &=  \liminf_{T \rightarrow \infty}   J^T(\xi,u(\cdot)), \\
 \overline J^{\infty}(\xi,u(\cdot)) &=  \limsup_{T \rightarrow \infty}  J^T(\xi,u(\cdot)).
 \label{averagerewardinftysup}
\end{align}
The optimal control problems in consideration here consist in maximizing these functionals with respect to the dilution rate $u(\cdot) \in {\cal U}(0,\infty)$, for any initial condition $\xi \in {\cal D}$.
More precisely, the value functions of these optimal control problems are
\begin{align} \label{Vinfty-liminf}
 \underline V^{\infty}(\xi) &= \sup \Big\{ \underline J^{\infty}(\xi,u(\cdot)) : u(\cdot) \in {\cal U}(0,\infty)  \Big\},   \\
 \overline  V^{\infty}(\xi) &= \sup \Big\{ \overline  J^{\infty}(\xi,u(\cdot)) : u(\cdot) \in {\cal U}(0,\infty)  \Big\}.
\label{Vinfty-limsup}
\end{align}

We need to consider the inferior and superior limits here as there exists controls for which the rewards \eqref{averagerewardinftyinf} and \eqref{averagerewardinftysup} may differ.
Indeed, this is the case for certain oscillating controls as can be seen in the example %\ref{counterex}
in the Appendix.
Nevertheless, we will show that the value functions \eqref{Vinfty-liminf} and \eqref{Vinfty-limsup} are in fact equal.
Moreover, we will connect these problems to the problem with a discounted reward when the discount factor goes to 0, as in \cite{grune1998relation}, and we will identify a set of controls that are optimal for all three problems.

To this end, we now define the following discounted reward, for a discount rate $\delta>0$
\begin{equation} \label{discntreward}
 J_{\delta}(\xi, u(\cdot)) = \delta \int_0^\infty  e^{-\delta t} \phi \big(s_{\xi,u}(t),z_{\xi,u}(t)\big) z_{\xi,u}(t) \,dt.
\end{equation}
This type of cost function is often used in problems related to economics for which the term $e^{-\delta t}$ represents a discount rate or a preference for the present \cite{carlson2012infinite}. In our setting, the use of this discounted reward can be seen as a preference for earlier rather than later production.
Here, the integral is rescaled with the discount factor $\delta$ in order to guarantee that, when we take the limit as $\delta$ goes to 0, the reward remains finite.

The value function of the optimal control problem for a given $\delta$ is then
\begin{align} \label{Vdelta}
 V_{\delta}(\xi) &= \sup \Big\{ J_{\delta}(\xi,u(\cdot)) : u(\cdot)  \in {\cal U}(0,\infty)  \Big\}.
\end{align}

Note that both average rewards \eqref{averagerewardinftyinf} and \eqref{averagerewardinftysup}, as well as the discounted reward \eqref{discntreward}, are well defined as the following Lemma shows.
\begin{lemma} \label{finterewards}
 For all  $\xi \in \cal D$, for all admissible controls $u(\cdot)\in {\cal U}(0,\infty)$ and for all $\delta >0$, the rewards $\underline J^{\infty}(\xi,u(\cdot))$, $\overline J^{\infty}(\xi,u(\cdot))$ and $J_{\delta}(\xi, u(\cdot))$ are uniformly bounded.
\end{lemma}
\begin{proof}
From the monotonicity properties of Assumption \ref{assum:growthfct}, we have that the function $z \mapsto \phi(s,z)$ is non increasing. for all $s>0$.
Thus, for all $t \geqslant 0$
\[
\phi(s_{\xi,u}(t),z_{\xi,u}(t)) \leqslant \overline \phi(0).
\]
The uniform boundedness of the rewards then follows from Lemma \ref{lemcompinv}.

\end{proof}

\subsection{ Relation Between Average and Discounted Biogas Production Problems}

We now show how the average and discounted biogas production problems are related when the discount factor $\delta$ goes to 0.

In the following, we will consider the discounted reward \eqref{discntreward} as a function of the trajectory $\zeta(\cdot)=\big( s_{\xi,u}(\cdot), z_{\xi,u}(\cdot) \big)$ instead of the control and with a slight abuse of notation, we will denote it as $J_{\delta}( \zeta(\cdot) )$.
Define the set valued map
\[
F(\zeta) := \bigcup_{u \in [0,u_{\max}]} f(\zeta,u)
\]
and consider the set of all forward trajectories of \eqref{eq:chemosz} with initial
condition $\xi$
%\begin{equation*}
% \mathcal{S}(\xi) := \Big\{ \big( s_{\xi,u}(\cdot), z_{\xi,u}(\cdot) \big) :
% u(\cdot)  \in {\cal U}(0,\infty)   \Big\}
%\end{equation*}
\begin{equation*}
 \mathcal{S}(\xi) := \Big\{ \zeta(\cdot) \in {\cal AC}([0,\infty),{\cal L}(\xi)) :
 \zeta(0)=\xi, \; \dot \xi(t) \in F(\xi(t)) \mbox{ a.e. } t \in [0,\infty) \Big\}
\end{equation*}
where ${\cal AC}([0,\infty),{\cal L}(\xi))$ denotes the set of absolutely continuous functions from $[0,\infty)$ to ${\cal L}(\xi)$.
We recall from the Filippov Selection Theorem (see for instance \cite{Vinter2000}) that the optimal control problem \eqref{Vdelta} is equivalent to the optimization problem on $\mathcal{S}(\xi)$,
\[
 V_{\delta}(\xi) = \sup \Big\{ J_{\delta}( \zeta(\cdot) ) : \zeta(\cdot)  \in \mathcal{S}(\xi)    \Big\}.
\]

We now specify the topology that we will use to study the limit of the discounted biogas production problem when the discount factor $\delta$ goes to 0.

\begin{defn} \label{def:topoWeightedW11}
For $b>0$, we denote by $L^1\big(0,\infty;\R^2, e^{-bt}dt\big)$ the weighted Lebesgue space of measurable functions $y(\cdot)$ from $[0,\infty)$ to $\R^2$ such that
\[
 \int_0^{\infty} ||y(t)|| e^{-bt}dt < \infty
\]
and we denote $W^{1,1}\big(0,\infty;\R^2, e^{-bt}dt\big)$ the weighted Sobolev space of measurable functions $y(\cdot)$ satisfying
\[
 y(\cdot) \in L^1\big(0,\infty;\R^2, e^{-bt}dt\big) \text{ and } \dot y(\cdot) \in L^1\big(0,\infty;\R^2, e^{-bt}dt\big).
\]
We consider the topology on $W^{1,1}\big(0,\infty;\R^2, e^{-bt}dt\big)$ for which a sequence $y_n(\cdot)$ converges to $y(\cdot)$ if and only if
\begin{itemize}
 \item[-] $y_n(\cdot)$ converges uniformly to $y(\cdot)$ on compact intervals,
 \item[-] $\dot y_n(\cdot)$ converges weakly to $\dot y(\cdot)$ in $L^1\big(0,\infty;\R^2, e^{-bt}dt\big)$.
\end{itemize}
\end{defn}

Now, we define the notion of $\Gamma-$limit in our context (see \cite{dal2012introduction} for further details).

\begin{defn} \label{def:gammalim}
For a given initial condition $\xi \in \cal D$ and trajectory $\zeta(\cdot) \in{\cal S}(\xi)$, the \emph{$\Gamma-$lower limit} and \emph{$\Gamma-$upper limit} of $J_{\delta}(\cdot)$ are
\begin{align*}
 \Gamma- \liminf_{\delta \rightarrow 0} J_{\delta}( \zeta(\cdot) ) &= \sup_{ {\cal V \in  N}(\zeta(\cdot)) } \liminf_{\delta \rightarrow 0} \inf_{\eta(\cdot) \in {\cal V}} J_{\delta}( \eta(\cdot) ) \\
  \Gamma- \limsup_{\delta \rightarrow 0} J_{\delta}( \zeta(\cdot) ) &= \sup_{{\cal V \in  N}(\zeta(\cdot)) } \limsup_{\delta \rightarrow 0} \inf_{\eta(\cdot) \in {\cal V}} J_{\delta}( \eta(\cdot) ).
\end{align*}
Here, we denote ${\cal N}(\zeta(\cdot)) $ the set of all open neighborhoods of $\zeta(\cdot)$ of the topology on $W^{1,1}\big(0,\infty;\R^2, e^{-bt}dt\big)$ given in Definition \ref{def:topoWeightedW11}.
If both of these limits coincide, then the \emph{$\Gamma-$limit} of $J_{\delta}( \cdot)$ is
\[
 \Gamma- \lim_{\delta \rightarrow 0} J_{\delta}( \zeta(\cdot) ) = \Gamma- \liminf_{\delta \rightarrow 0} J_{\delta}( \zeta(\cdot) ) = \Gamma- \limsup_{\delta \rightarrow 0} J_{\delta}( \zeta(\cdot) ).
\]
\end{defn}

We now show that this $\Gamma-$limit is well defined, as well as the associated optimal control problem.

\begin{proposition} \label{prop:limVdelta}
For all $\xi \in \cal D$ and for all trajectories $\zeta(\cdot) \in{\cal S}(\xi)$, the $\Gamma-$limit of $J_{\delta}(\cdot)$ exists and we denote it as
\[
  J_0(\zeta(\cdot)) := \Gamma- \lim_{\delta \rightarrow 0} J_{\delta}(\zeta(\cdot) ).
\]
In addition, for all $\delta >0$, the suprema are attained
\[
 V_{\delta}(\xi) = \max_{u(\cdot)} J_{\delta}(\zeta(\cdot))
\]
and these maxima converge as $\delta$ goes to 0, pointwise in $\xi$,
\begin{equation} \label{def:V0}
 V_0(\xi) :=\max_{\zeta(\cdot)}  J_{0}(\zeta(\cdot))  =  \lim_{\delta \rightarrow 0} V_{\delta}(\xi).
\end{equation}
Finally, if $\zeta_{\delta}(\cdot)$ is an optimal trajectory for \eqref{Vdelta}, i.e. if $V_{\delta}(\xi)=J_{\delta}(\zeta_{\delta}(\cdot))$, and if $\zeta_{\delta}(\cdot)$ converges to $\zeta_0(\cdot)$ in $ {\cal S}(\xi)$, then $\zeta_0(\cdot)$ is an optimal control for \eqref{def:V0} and
\[
 V_0(\xi) = J_0(\zeta_0(\cdot) ) = \lim_{\delta \rightarrow 0} J_{\delta}(\zeta_{\delta}(\cdot)).
\]
\end{proposition}

\begin{proof}
First, we show that for all trajectories $\zeta(\cdot) = \big(s(t),z(t)\big) \in{\cal S}(\xi)$ and for $\delta$ small enough, $\delta \mapsto J_{\delta}(\zeta(\cdot))$ is increasing. We can write this function as
\[
	J_{\delta}(\zeta(\cdot)) = \delta \int_0^{+\infty} e^{-\delta t} g(t) dt,
\]
with $g(t) :=\phi \big(s(t),z(t)\big) z(t)$, which is bounded and positive (Lemma \ref{lemcompinv}),
\[
0< m < g(t) < M < \infty, \quad \forall t\geqslant 0.
\]
Then, we have
\[
	\frac{\partial}{\partial \delta} J_{\delta}(\zeta(\cdot)) = \int_0^{+\infty} e^{-\delta t} g(t) dt - \delta \int_0^{+\infty} te^{-\delta t} g(t) dt
\]
so that for $T>0$,
\[
	\frac{\partial}{\partial \delta} J_{\delta}(\zeta(\cdot)) > \frac{m}{\delta} \left( 1 - e^{-\delta T} \right) - \frac{M}{\delta}.
\]
Now, since $e^{-\delta T} = 1 - \delta T + o(\delta)$, there exists $\bar \delta >0$ such that, for all $\delta < \bar \delta$, $\frac{m}{\delta} \big( 1 - e^{-\delta T})> \frac{mT}{2}$.
Then, taking $T> \frac{2M }{m \delta}$, we conclude that $\delta \mapsto J_{\delta}(\zeta(\cdot))$ is increasing for $\delta < \bar \delta$.

Next, recall that for all initial conditions and all trajectories, $\delta \mapsto J_{\delta}(\zeta(\cdot))$ is uniformly bounded (Lemma \ref{finterewards}).
Finally, since $J_{\delta}(\cdot)$ is continuous with respect to $\zeta(\cdot)$, we can use \cite[Proposition 5.7]{dal2012introduction} to get the $\Gamma-$convergence as $\delta$ goes to 0.

To show that the suprema are attained and that they converge, it is sufficient to show that there exists a countably compact set on which the suprema are attained for all $\delta$ \cite[Theorem 7.4]{dal2012introduction}.
The set $\mathcal{S}(\xi)$ is clearly independent of $\delta$
so that we now need to show that $\mathcal{S}(\xi)$ is countably compact for the topology on $W^{1,1}\big(0,\infty;\R^2, e^{-bt}dt\big)$ given in Definition \ref{def:topoWeightedW11}. However, since $W^{1,1}\big(0,\infty;\R^2, e^{-bt}dt\big)$ is a metric space, any compact set is countably compact, so we only need to prove the compactness of $\mathcal{S}(\xi)$.

For each $\xi \in \cal D$ we set
\[
 F_{\xi}(\zeta) := F \big( P_{\cal L(\xi)} (\zeta) \big)
\]
where $P_{\cal L(\xi)}$ is the projection on the convex set $\cal L(\xi)$.
Then $F_{\xi}$ has linear growth, so that we can define
\[
 c  = \sup_{\zeta \in \text{Dom}(F_{\xi})} \frac{||F_{\xi}(\zeta)||}{||\zeta||+1}
\]
where $||F_{\xi}(\zeta)|| := \sup_{\eta \in F_{\xi}(\zeta)} ||\eta||$.
Note that $F$ is upper semi-continuous and has compact non-empty convex images
(such a map is known as a Marchaud map \cite{aubin2009viability}).
With this, the set $\mathcal{S}(\xi)$ is the set of absolutely continuous solutions of the differential inclusion
\[
 \dot \zeta(t) \in F_{\xi}(\zeta(t)), \qquad \zeta(0)=\xi.
\]
We can therefore use \cite[Theorem 3.5.2]{aubin2009viability} to establish that $\mathcal{S}(\xi)$ is compact in $W^{1,1} \big(0,\infty;\R^2, e^{-bt}dt\big)$ for $b>c$.

\end{proof}

We now relate the average and discounted biogas production problems.

\begin{proposition} \label{prop:grune}
For all $\xi \in \cal D$ we have
 \[
    \underline V^{\infty}(\xi) \leqslant  V_{0}(\xi) \leqslant \overline V^{\infty}(\xi).
 \]
\end{proposition}

\begin{proof}
We adapt here results of \cite{grune1998relation} given for minimization problems to maximization problems by changing the sign of the reward.
We give here the main steps for the first inequality, the second is obtained similarly.
First, \cite[Lemma 3.3]{grune1998relation} gives
 \[
    \sup_{u(\cdot)} \liminf_{T\rightarrow \infty} J^T(\xi , u(\cdot)) = \lim_{T\rightarrow \infty} \sup_{u(\cdot)} \inf_{\tau \geqslant T} J^{\tau}(\xi, u(\cdot))
 \]
and then  \cite[Corollary 3.5]{grune1998relation} states that for all $T>0$, all $\varepsilon > 0$ we have
 \[
  \sup_{u(\cdot)} \inf_{\tau \geqslant T} J^{\tau}(\xi, u(\cdot)) - \varepsilon \leqslant  V_{\delta} (\xi)
 \]
for all small $\delta$.
Taking the limit as $T \rightarrow  \infty$ and $\delta \rightarrow 0 $ gives the result.

\end{proof}

\subsection{Solution of Optimal Control Problems}

We now solve the optimal control problems \eqref{Vinfty-liminf} and \eqref{Vinfty-limsup} and show that their value functions are equal to the limit \eqref{def:V0} of the discounted problem.
We start by determining an upper bound for the value functions and then we will exhibit controls that attain this bound.

\begin{proposition} \label{prop:majV}
For all initial conditions $\xi \in \cal D$
 \[
   \overline V^{\infty}(\xi)  \leqslant \max_{s\in (0,s_{in})} \phi(s,1).
 \]
\end{proposition}

\begin{proof}
With the monotonicity properties of $\mu(\cdot,\cdot)$ of Assumption \ref{assum:growthfct}, we have that $z\mapsto \phi(s,z)$ is non increasing and $z\mapsto \phi(s,z)z$ is non decreasing.
This implies that
\begin{align} \label{ineqphi}
 \phi (s, \max(z_0,1)) & \leqslant  \phi (s,z) \leqslant  \phi (s, \min(z_0,1))
\end{align}
and
\begin{align} \label{ineqphiz}
 \phi (s, \min(z_0,1))  \min(z_0,1) \leqslant  \phi (s,z) z \leqslant  \phi (s, \max(z_0,1)) \max(z_0,1).
\end{align}
First, we consider the case when $z_0 \leqslant 1$. For any control $u(\cdot)$, we have
\begin{align*}
 J^T(\xi, u(\cdot)) & \leqslant \frac{1}{T} \int_0^T \phi \big(s(t), \max(z_0,1)\big) \max(z_0,1) \, dt \\
 & \leqslant \max_{s\in (0,s_{in})} \phi(s,1) = \overline \phi(1).
\end{align*}
Taking the upper limit as $T$ goes to infinity and the supremum with respect to $u(\cdot)$ we get the result.

Next, for $z_0 \geqslant 1$, we have
\begin{align*}
 J^T(\xi, u(\cdot)) & \leqslant \frac{1}{T} \int_0^T \phi \big(s(t), \min(z_0,1)\big) z(t) \, dt \\
 & \leqslant \max_{s\in (0,s_{in})} \phi(s,1)  \frac{1}{T} \int_0^T z(t) \, dt.
\end{align*}
Using Lemma \ref{limz} we get that $\overline J^\infty (\xi, u(\cdot)) \leqslant \overline \phi(1)$ and we conclude taking the supremum with respect to $u(\cdot)$.

\end{proof}

Note that the existence of a maximum of $s \mapsto \phi(s,1)= \mu(s,s_{in}-s)(s_{in}-s)$ on $(0,s_{in})$ follows from Assumption \ref{assum:growthfct}.
We will denote a substrate level at which such a maximum is attained as
\begin{equation*}
 \bar s  = \argmax_{s\in (0,s_{in})} \phi(s,1)
\end{equation*}

\begin{proposition} \label{prop:asymopt}
 For any initial condition $\xi \in \cal D$, any control $\overline u(\cdot) \in {\cal U}(0,\infty)$ that drives the system asymptotically to the state $(\bar s, 1)$ is optimal for problems \eqref{Vinfty-liminf}, \eqref{Vinfty-limsup} and \eqref{def:V0}.
 We then have
 \begin{equation} \label{eq:ValfctEqual}
  \underline V^{\infty}(\xi) =  V_{0}(\xi) = \overline V^{\infty}(\xi) = \phi(\bar s,1) = \overline \phi(1).
 \end{equation}
\end{proposition}

\begin{proof}
The continuity of $\phi$ implies that for all $\varepsilon > 0$, there exists a time $t_{\varepsilon}\geqslant0$ such that, for all $t\geqslant t_{\varepsilon}$,
\begin{equation} \label{eq:contiPhi}
\left| \phi \big(s_{\xi,\bar u}(t),z_{\xi,\bar u}(t)\big)z_{\xi,\bar u}(t)-\phi \big(\bar s,1\big)\right|< \varepsilon.
\end{equation}
Since $s_{\xi,\bar u}(\cdot)$ and $z_{\xi,\bar u}(\cdot)$ take values in the compact set $\cal L(\xi)$ \eqref{Lxi}, there is a constant $M_{\xi}>0$ such that, for all $t\geqslant0$,
\begin{equation} \label{eq:phiBorne}
 \left| \phi \big(s_{\xi,\bar u}(t),z_{\xi,\bar u}(t)\big)z_{\xi,\bar u}(t) \right| < M_{\xi}.
 \end{equation}
Then, for all $T\geqslant t_{\varepsilon}$, from \eqref{eq:contiPhi} and \eqref{eq:phiBorne}
\begin{align*}
 \Big| J^T(\xi, \overline u (\cdot) )  -\phi \big(\bar s,1\big) \Big|
 &\leqslant  \frac{1}{T}  \int_0^{t_{\varepsilon}} \left| \phi \big(s_{\xi,\bar u}(t),z_{\xi,\bar u}(t)\big)z_{\xi,\bar u}(t)-\phi \big(\bar s,1\big) \right| \,dt \\
 & \qquad + \frac{1}{T} \int_{t_{\varepsilon}}^T \left| \phi \big(s_{\xi,\bar u}(t),z_{\xi,\bar u}(t)\big)z_{\xi,\bar u}(t)-\phi \big(\bar s,1\big)\right| \,dt   \\
 &<  \frac{2 M_{\xi} t_{\varepsilon}}{T}   +\left(1 - \frac{ t_{\varepsilon}}{T}\right) \varepsilon
\end{align*}
and we have
\[
 \underline J^{\infty}(\xi,\overline u(\cdot)) = \overline J^{\infty}(\xi,\overline u(\cdot)) = \phi(\bar s,1).
\]
Using Propositions \ref{prop:grune} and \ref{prop:majV}, we get the equality of value functions \eqref{eq:ValfctEqual} and deduce the optimality of $\overline u(\cdot)$ for both average biogas production problems \eqref{Vinfty-liminf} and \eqref{Vinfty-limsup}.
We proceed similarly to get
\begin{align*}
 \Big| J_{\delta}(\xi, \overline u (\cdot) )  -\phi \big(\bar s,1\big) \Big|
 &<   2 M_{\xi} \int_0^{t_{\varepsilon}} \delta e^{-\delta t} \,dt  + \varepsilon \int_{t_{\varepsilon}}^{\infty} \delta e^{-\delta t} \,dt   \\
 &<  2 M_{\xi} \left(1- e^{-\delta t_{\varepsilon}} \right)  - \varepsilon  e^{-\delta t_{\varepsilon}}
\end{align*}
and we have
\[
 J_{0}(\xi, \overline u (\cdot) ) = \phi(\bar s,1).
\]
Then, Proposition \ref{prop:grune} implies that $\overline u(\cdot)$ is also optimal for problem \eqref{def:V0}.

\end{proof}

% The previous proposition shows that the value functions of the considered problems are equal to the biogas flowrate $\phi(s,z)z$ for $s=\bar s$ and $z=1$ if this state can be attained.
With Lemma \ref{limz}, we know that all persistently exciting admissible controls make $z(\cdot)$ converge to 1, and from Lemma \ref{lem:sStarReachable}, we know that the feedback $ \psi_{s^*}$ defined in \eqref{def:MRAPfeedback} with $s^*= \bar s$ guarantees that $s(\cdot)$ reaches $\bar s$.
Then, from the previous Proposition we have the following result.

\begin{proposition}
 For any initial condition $\xi \in \cal D$  satisfying Assumption \ref{assum:umax}, the most rapid approach feedback to $\bar s$, defined in \eqref{def:MRAPfeedback} and denoted $\psi_{\bar s}$, is optimal for both average production problems \eqref{Vinfty-liminf} and \eqref{Vinfty-limsup} and for the limit \eqref{def:V0} of the discounted production problem.
\end{proposition}

Clearly, there is not a unique optimal control for the infinite horizon problems that we have considered.
For example, in the case of a growth function that depends only on the substrate and that is monotone (such as the Monod growth function), the constant control $u=\mu(\bar s)$ can also drive the system to the state $(\bar s, 1)$.
Nonetheless, for the control $\psi_{\bar s}$, we are able to state in the next section an estimation of the sub-optimality for the finite horizon problem.

% our interest in the feedback $\psi_{\bar s}$ lies in its simplicity, in particular in the implementation of such a control.

%%%%%%%%%%%%%%%%%%%%%%%%%%%%%%%%%%%%%%%%%%%%%%%%%%%%%%%%%%%%%%%%%%%%%%%%%%%%%%%%%%%%%%%%%%%%%%%

\section{Finite Horizon and Sub-optimal Controls}
\label{secfinite}

We now examine the problem of maximizing biogas production over a finite horizon for a time interval $[t_0,T]$ where $T$ is fixed.
For this we consider the following reward
\begin{equation} \label{eq:finitehorizonreward}
J(t_{0},\xi,u(\cdot))=\int_{t_{0}}^T \phi \big( s_{t_0,\xi,u}(t),z_{t_0,\xi,u}(t) \big) z_{t_0,\xi,u}(t) \,dt
\end{equation}
where we recall that $\big( s_{t_0,\xi,u}(\cdot), z_{t_0,\xi,u}(\cdot) \big) $ is the solution of \eqref{eq:chemosz} with control $u(\cdot) \in {\cal U}(t_0,T)$ and initial condition $\xi \in {\cal D}$.
The optimal control problem consists in maximizing this functional with respect to the dilution rate, so that the associated value function is
\begin{equation} \label{eq:VleOrig-finithorizon}
 V(t_{0},\xi) =   \sup \Big\{  J(t_{0},\xi,u(\cdot)) : u(\cdot) \in {\cal U}(t_0,T)  \Big\}.
\end{equation}

We also consider {\em auxiliary} optimal control problems, which consist in maximizing the cost, for a given $z_1\in[ \min(z_0,1) , \max(z_0, 1) ]$,
\begin{align} \label{def:aux-reward}
 J_{z_1}(t_{0},\xi,u(\cdot)) & =  \int_{t_{0}}^T \phi(s_{t_0,\xi,u}(t),z_1)\,dt
\end{align}
for the same dynamics \eqref{eq:chemosz}.
The value functions of these auxiliary problems are then defined as
\begin{align}
\label{eq:vfaux}
W_{z_1}(t_{0},\xi) & = \sup \Big\{  J_{z_1}(t_{0},\xi,u(\cdot)) : u(\cdot) \in {\cal U}(t_0,T)  \Big\}.
\end{align}
The resolution of these auxiliary problems will be presented in Section \ref{section:solveaux}.

We now show that the value functions of the original problem \eqref{eq:VleOrig-finithorizon} and the auxiliary problems \eqref{eq:vfaux} are related.

\begin{proposition} \label{prop:vleFctFrame}
For all $\xi \in \cal D$, $t_0<T$ and any $z_1 \in [ \min(z_0,1) , \max(z_0, 1) ]$, we have the following frame for the value function $V$ of the original problem
\begin{equation}
 \label{eq:vleFctFrame}
 \min(z_{0},1)  W_{z_1}(t_{0},\xi) \leqslant V(t_{0},\xi) \leqslant \max(z_{0},1) W_{z_1}(t_{0},\xi).
\end{equation}
\end{proposition}

\begin{proof}
We start with the case $z_0 \leqslant 1$. For a given control $u(\cdot) \in {\cal U}(t_0,T)$, we define the following time
\[
  t_1 = \inf \left\{ t \geqslant t_0 : z_{t_0,\xi,u}(t) = z_1 \right\} \wedge T
\]
which it is well defined since $z_{t_0,\xi,u}(\cdot)$ is monotonous.
Then, for $t_0 \leqslant t \leqslant t_1$ we have $z_0 \leqslant z_{t_0,\xi,u}(t) \leqslant z_1 \leqslant 1$
and with the monotonicity properties of $\mu(\cdot,\cdot)$ of Assumption  \ref{assum:growthfct} we have
\[
 \phi(s_{t_0,\xi,u}(t),z_1) z_0 \leqslant \phi(s_{t_0,\xi,u}(t),z_{t_0,\xi,u}(t))z_{t_0,\xi,u}(t) \leqslant \phi(s_{t_0,\xi,u}(t),z_1) z_1.
\]
Next, for $t_1 \leqslant t \leqslant T$ we have $z_0 \leqslant z_1 \leqslant z_{t_0,\xi,u}(t)  \leqslant 1$
and
\[
 \phi(s_{t_0,\xi,u}(t),z_1) z_1 \leqslant \phi(s_{t_0,\xi,u}(t),z_{t_0,\xi,u}(t))z_{t_0,\xi,u}(t) \leqslant \phi(s_{t_0,\xi,u}(t),z_1).
\]
Combining these inequalities we get
\begin{align*}
 \int_{t_{0}}^{t_1} \phi(s_{t_0,\xi,u}(t),z_1) z_0 \,dt + & \int_{t_1}^T \phi(s_{t_0,\xi,u}(t),z_1)z_1 \,dt
  \leqslant J(t_0,\xi,u(\cdot) )  \\
 & \leqslant
 \int_{t_{0}}^{t_1} \phi(s_{t_0,\xi,u}(t),z_1) z_1 \,dt + \int_{t_1}^T \phi(s_{t_0,\xi,u}(t),z_1)\,dt.
\end{align*}
Now, since $z_0 \leqslant z_1 \leqslant 1$ we have
\begin{equation*}
 z_0 J_{z_1}(t_{0},\xi,u(\cdot)) \leqslant J(t_{0},\xi,u(\cdot)) \leqslant J_{z_1}(t_{0},\xi,u(\cdot)).
\end{equation*}
For the case $z_0 \geqslant 1$, we proceed in a similar way to get
\begin{equation*}
  J_{z_1}(t_{0},\xi,u(\cdot)) \leqslant J(t_{0},\xi,u(\cdot)) \leqslant z_0 J_{z_1}(t_{0},\xi,u(\cdot)).
\end{equation*}
We conclude by taking the supremum over all admissible controls.

\end{proof}

The interest of the previous frames on the value functions is that it allows to find controls for which we have an estimation of sub-optimality for the original problem.

\begin{proposition}
 For all $\xi \in \cal D$ and all $t_0<T$, any optimal control $ u^\star_{z_1}(\cdot)$ for the reward $J_{z_1}(t_0,\xi,\cdot)$ guarantees a (sub-optimal) value for the original criterion $J(t_0,\xi,\cdot)$ that satisfies
\begin{equation}
 \label{eq:subOptFrame}
\min(z_{0},1) W_{z_1}(t_{0}, \xi)  \leqslant  J(t_{0},\xi, u^\star_{z_1}(\cdot) ) \leqslant \max(z_{0},1) W_{z_1}(t_{0},\xi)
\end{equation}
and we have the following estimation of the value function $V$
\begin{equation}
 \label{eq:estSubOpt}
  V(t_0, \xi) - J(t_0, \xi, u^\star_{z_1}(\cdot)) \leqslant  |1 - z_0 | W_{z_1}(t_0, \xi).
\end{equation}
\end{proposition}

\begin{proof}
From the proof of Proposition \ref{prop:vleFctFrame}, for any control $ u(\cdot) \in {\cal U}(t_0,T)$, we have
\begin{equation*}
  \min(z_0,1) J_{z_1}(t_{0},\xi,u(\cdot)) \leqslant J(t_{0},\xi,u(\cdot)) \leqslant \max(z_0,1) J_{z_1}(t_{0},\xi,u(\cdot)).
\end{equation*}
Evaluating this for any optimal control $ u^\star_{z_1}(\cdot)$ for the reward $J_{z_1}(t_0,\xi,\cdot)$ gives the sub-optimality frame \eqref{eq:subOptFrame}.
The sub-optimality estimation \eqref{eq:estSubOpt} then follows from \eqref{eq:vleFctFrame} and \eqref{eq:subOptFrame}.

\end{proof}

\subsection{ Resolution of Auxiliary Problems } \label{section:solveaux}

In order to obtain sub-optimal controls for problem \eqref{eq:VleOrig-finithorizon} we now need to solve the auxiliary problem \eqref{eq:vfaux} for a given $z_1 \in [ \min(z_0,1) , \max(z_0, 1) ]$.
The optimal control of this auxiliary problem is an autonomous feedback, even though the horizon is fixed and finite.
It is similar to the optimal feedback for the infinite horizon problem $\psi_{\bar s}$, defined in \eqref{def:MRAPfeedback}, and it drives the system towards a maximizer of $s \mapsto \phi(s,z_1)$ but now, this maximizing substrate level depends on $z_1$.

We first need an assumption on the uniqueness of a maximum of $\phi(\cdot,z_1)$.

\begin{assumption} \label{assum:maxphi-z1}
  For each $z_1 \geqslant 0 $, the function $s \mapsto \phi(s,z_1)$ admits a unique maximum on  $(0,s_{in})$, and we denote the substrate level at which this maximum is attained as
  \begin{equation} \label{eq:sbar-z1}
   \bar s(z_1)  = \argmax_{s\in (0,s_{in})} \phi(s,z_1).
  \end{equation}
\end{assumption}

Note that implies that $s \mapsto \phi(s,z_1)$ is increasing on $(0,\bar{s}(z_1)]$ and decreasing on $[\bar{s}(z_1),s_{in})$.

\begin{proposition}
 For all $\xi \in \cal D$ satisfying Assumption \ref{assum:umax} and all $t_0<T$,  the most rapid approach feedback to $\bar s(z_1)$, defined in \eqref{def:MRAPfeedback} and denoted $\psi_{\bar s(z_1)}$,  is optimal for the auxiliary problem \eqref{eq:vfaux}.
\end{proposition}

\begin{proof}
We start with the case $s_0 \geqslant \bar{s}(z_1)$.
With the control $u=0$, the solution of \eqref{eq:chemosz} is such that $s_{t_0,\xi,0}(\cdot)$ is monotonic and non increasing.
Therefore there exists a time $t_{min}$, possibly larger than $T$, such that $s_{t_0,\xi,0}(t_{min})=\bar{s}(z_1)$ and then the solution with the feedback \eqref{def:MRAPfeedback} is, with $ t_* = \min(t_{min},T)$
\[
s_{t_0,\xi,\psi_{\bar s(z_1)}}(t)=
\begin{cases}
 s_{t_0,\xi,0}(t) \qquad & \mbox{if }  t_0 \leqslant t < t_*, \\
 \bar{s}(z_1) & \mbox{if }  t_* \leqslant t \leqslant T.
\end{cases}
\]
Next, for all $u \in [0,u_{max}]$ and for all $(s,z) \in {\cal L}(\xi) $,
\[
 -\mu(s,(s_{in}-s)z)(s_{in}-s)z \leqslant (s_{in}-s)u - \mu(s,(s_{in}-s)z)(s_{in}-s)z.
\]
By the theorem of comparison of solutions of scalar differential equations, this implies that $ s_{t_0,\xi,0}(t) \leqslant s_{t_0,\xi,u}(t)$, up to time $t_*$, for all controls $u(\cdot) \in {\cal U}(t_0,T)$.
Since $s \mapsto \phi(s,z_1)$ is decreasing on $[\bar{s}(z_1),s_{in})$, we have
\[
 \phi( s_{t_0,\xi,0}(t),z_1 ) \geqslant \phi( s_{t_0,\xi,u}(t),z_1 ).
\]
Finally, as $s \mapsto \phi(s,z_1)$ reaches its maximum at $\bar{s}(z_1)$ we get
\begin{align*}
 J_{z_1}(t_0,\xi,\psi_{\bar s(z_1)}) &= \int_{t_0}^{t_*} \phi( s_{t_0,\xi,0}(t),z_1 ) dt + \int_{t_*}^{T} \phi( \bar s(z_1),z_1 ) dt \\
 & \geqslant \int_{t_0}^{T} \phi( s_{t_0,\xi,u}(t),z_1 ) dt \\
 & = J_{z_1}(t_0,\xi,u).
\end{align*}
We now consider $ s_0 < \bar{s}$.
From Assumption \ref{assum:umax}, the feedback is admissible and we have
\[
 u_{\max} \geqslant \mu(s,(s_{in}-s)z)z \qquad \text{for all } (s,z) \in {\cal L}(\xi)
\]
Thus, with the control $u=u_{\max}$, the solution of \eqref{eq:chemosz} is such that $s_{t_0,\xi,u_{\max}}(\cdot)$ is monotone and non decreasing.
Therefore, there exists a time $t_{\max}$, possibly larger than $T$, such that $s_{t_0,\xi,u_{\max}}(t_{\max})=\bar{s}(z_1)$ and then the solution with the feedback \eqref{def:MRAPfeedback} is, with $ t_* = \min(t_{\max},T)$
\[
s_{t_0,\xi,\psi_{\bar s(z_1)}}(t)=
\begin{cases}
 s_{t_0,\xi,u_{max}}(t) \qquad & \mbox{if }  t_0 \leqslant t < t_*, \\
 \bar{s}(z_1) & \mbox{if }  t_* \leqslant t \leqslant T.
\end{cases}
\]
Next, for all $u \in [0,u_{max}]$ and for all $(s,z) \in {\cal L}(\xi) $
\[
 (s_{in}-s)(u_{max} - \mu(s,(s_{in}-s)z)z) \geqslant (s_{in}-s)(u - \mu(s,(s_{in}-s)z)z)
\]
and this implies that $ s_{t_0,\xi,u_{\max}}(t) \geqslant s_{t_0,\xi,u}(t)$, up to time $t_*$, for all controls $u(\cdot) \in {\cal U}(t_0,T)$.
Since $s \mapsto \phi(s,z_1)$ is increasing on $(0,\bar{s}(z_1)]$, we have
\[
 \phi( s_{t_0,\xi,u_{max}}(t),z_1 ) \geqslant \phi( s_{t_0,\xi,u}(t),z_1 ).
\]
Finally, since $s \mapsto \phi(s,z_1)$ reaches its maximum at $\bar{s}(z_1)$, we get
\begin{align*}
 J_{z_1}(t_0,\xi,\psi_{\bar s(z_1)}) &= \int_{t_0}^{t_*} \phi( s_{t_0,\xi,u_{max}}(t),z_1 ) dt + \int_{t_*}^{T} \phi( \bar s(z_1),z_1 ) dt \\
 & \geqslant \int_{t_0}^{T} \phi( s_{t_0,\xi,u}(t),z_1 ) dt \\
 & = J_{z_1}(t_0,\xi,u).
\end{align*}

\end{proof}

%%%%%%%%%%%%%%%%%%%%%%%%%%%%%%%%%%%%%%%%%%%%%%%%%%%%%%%%%%%%%%%%%%%%%%%%%%%%%%%%%%%%%%%%%%%%%%%

\section{Application to Particular Growth Functions}
\label{secappli}

The controls that we have considered up to now are all most rapid approach feedbacks to $\bar s(z_1)$, with $z_1 \in [\min(z_0,1),\max(z_0,1)]$, and this leads to the question of which is best in terms of biogas production.
It turns out that it depends on the initial conditions and the horizon considered.

Indeed, we know that for an infinite horizon, the feedback $\psi_{\bar s(z_1)}$ with $z_1=1$ is optimal and we can then expect that when the horizon is large, the best of the considered feedbacks would be for $z_1$ close to 1.
On the other hand, when the horizon is small, the feedback $\psi_{\bar s(z_0)}$ would seem to be the best option since this strategy consists in remaining close to the maximum of the biogas flow rate corresponding to the initial condition, whereas another feedback could drive the system away, towards another maximizing state but that can not be reached in time.

In this section, we apply our main results to the most common growth functions and explore with numerical simulations the question of determining the best feedback $\psi_{\bar s(z_1)}$ for a given initial condition and final time.
In particular, we will work with the Monod function
\begin{equation} \label{monod}
 \mu_M(s)=\frac{\mu_{max}s}{K_s+s}
\end{equation}
the Haldane function
\begin{equation} \label{haldane}
 \mu_H(s) = \frac{\bar \mu s }{K_s + s + \frac{s^2}{K_{i}}}
\end{equation}
and the Contois function
\begin{equation} \label{contois}
 \mu_C(s,x)=\frac{\mu_{max}s}{K_s x+s}
\end{equation}
where $\mu_{max}$, $\bar \mu$, $ K_s $ and $K_i $ are positive numbers.
We shall see later that these functions satisfy our assumptions (Lemma \ref{lem:muAssump}).

First, note that the Monod and Haldane functions only depend on the substrate, so that in this case, the maximizers $\bar s(z_1)$, defined in \eqref{eq:sbar-z1}, are all equal to $\bar s(1) = \bar s$, for all $z_1 \in [\min(z_0,1) , \max(z_0,1) ]$.
We illustrate the associated feedback $\psi_{\bar s}$ for a Haldane function with a graph of the state space trajectories in Figure \ref{fig:sbar(1)}.
The case of a Monod function leads to a similar dynamical behavior and the only major difference is the value of $\bar s$.

From now on we will only consider the Contois growth function, for which we plot the trajectories in state space obtained with the feedback $\psi_{\bar s(z_0)}$ in Figure \ref{fig:sbar(z0)}.

\begin{figure}
 \begin{center}
 \includegraphics[width=0.8\textwidth]{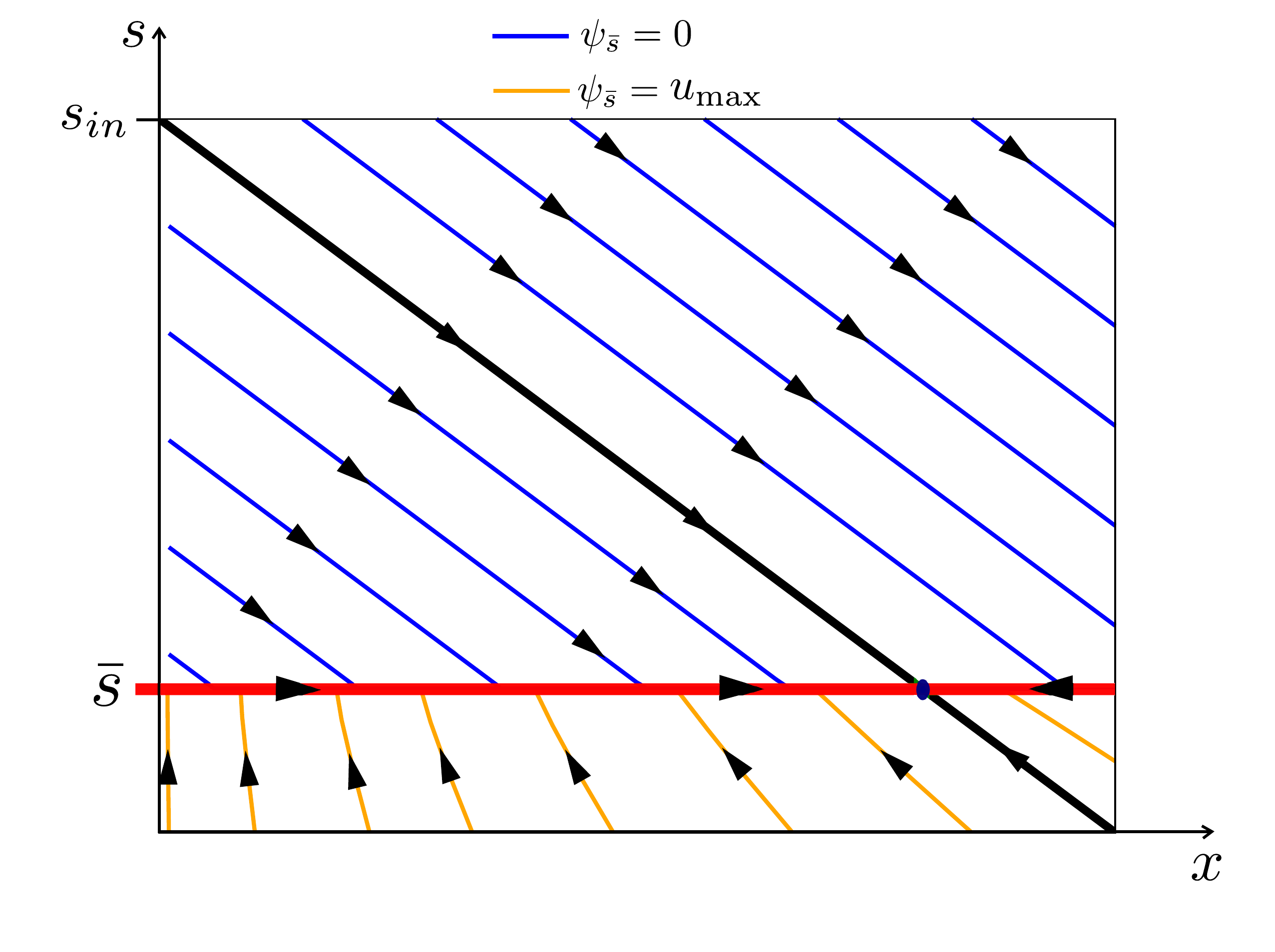}
  \caption{State space trajectories with feedback $\psi_{\bar s}$. The black line represents the invariant set $\{(x,s) :  x+s = s_{in} \}$.
  Haldane growth function ($\bar \mu=0.74, K_s=9.28, K_i=256$) with $s_{in}=100$, $u_{\max}=3$.}
  \label{fig:sbar(1)}
 \end{center}
\end{figure}

\begin{figure}
 \begin{center}
 \includegraphics[width=0.8\textwidth]{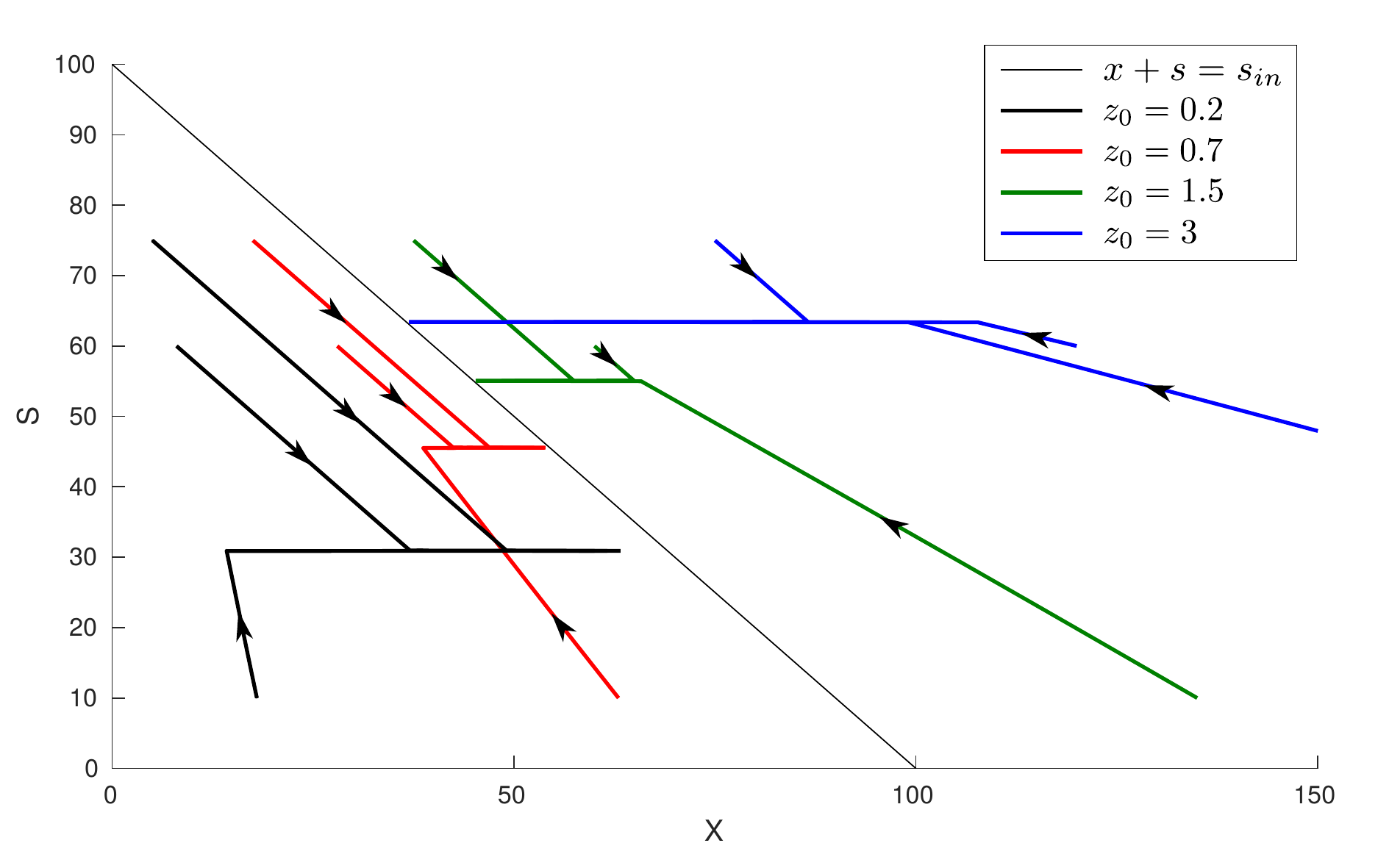}
  \caption{State space trajectories with feedback $\psi_{\bar s(z_0)}$ for $z_0 \in \{ 0.2 , 0.7 , 1.5 , 3\}$ and $s_0 \in \{ 10 , 60 , 75\}$.
   The color and type of line indicates the value of $z_0$.
  Contois growth function ($\mu_{max}=0.74, K_s=1$) with $s_{in}=100$, $u_{\max}=1.5$.}
  \label{fig:sbar(z0)}
 \end{center}
\end{figure}

To determine which of the feedbacks $\psi_{\bar s(z_1)}$ is the best, we now compute the associated reward for a range of values of $z_1 \in [\min(z_0,1) , \max(z_0,1) ]$ and of final times for a given initial condition. In order to easily identify the maximum of $J(\xi,\psi_{\bar s(z_1)}(\cdot))$ with respect to $z_1$, we normalize the average reward \eqref{eq:avgReward} by computing
\[
 J_N(T,z_1) = \frac{ J^T(\xi,\psi_{\bar s(z_1)}(\cdot)) - \min_{y} J^T(\xi,\psi_{\bar s(y)}(\cdot)) }{ \max_{y} J^T(\xi,\psi_{\bar s(y)}(\cdot)) - \min_{y} J^T(\xi,\psi_{\bar s(y)}(\cdot)) }
\]
where the minimum and maximum are taken for $y\in [\min(z_0,1) , \max(z_0,1) ]$.
Hence, for each final time $T$, the maximum reward is achieved for $z_1$ such that $J_{N}(T,z_1) = 1$ and the minimum when $J_{N}(T,z_1) = 0$.

Figure \ref{fig:normalReward-z0<1} shows a case when $z_0 <1$ and  Figure \ref{fig:normalReward-z0>1} is an example of $z_0 >1$.
We can see clearly that for small final times, the maximum is attained for a value of $z_1$ close to $z_0$ and that for $z_1=1$ the reward is the smallest.
However, as the final time increases, the value of $z_1$ for which the reward is maximum approaches 1, and with the feedback $\psi_{\bar s(z_0)}$ the reward is the smallest.
In particular, we can see that the best of the feedbacks $\psi_{\bar s(z_1)}$ depends on the final time.

\begin{figure}
 \begin{center}
 \includegraphics[width=0.8\textwidth]{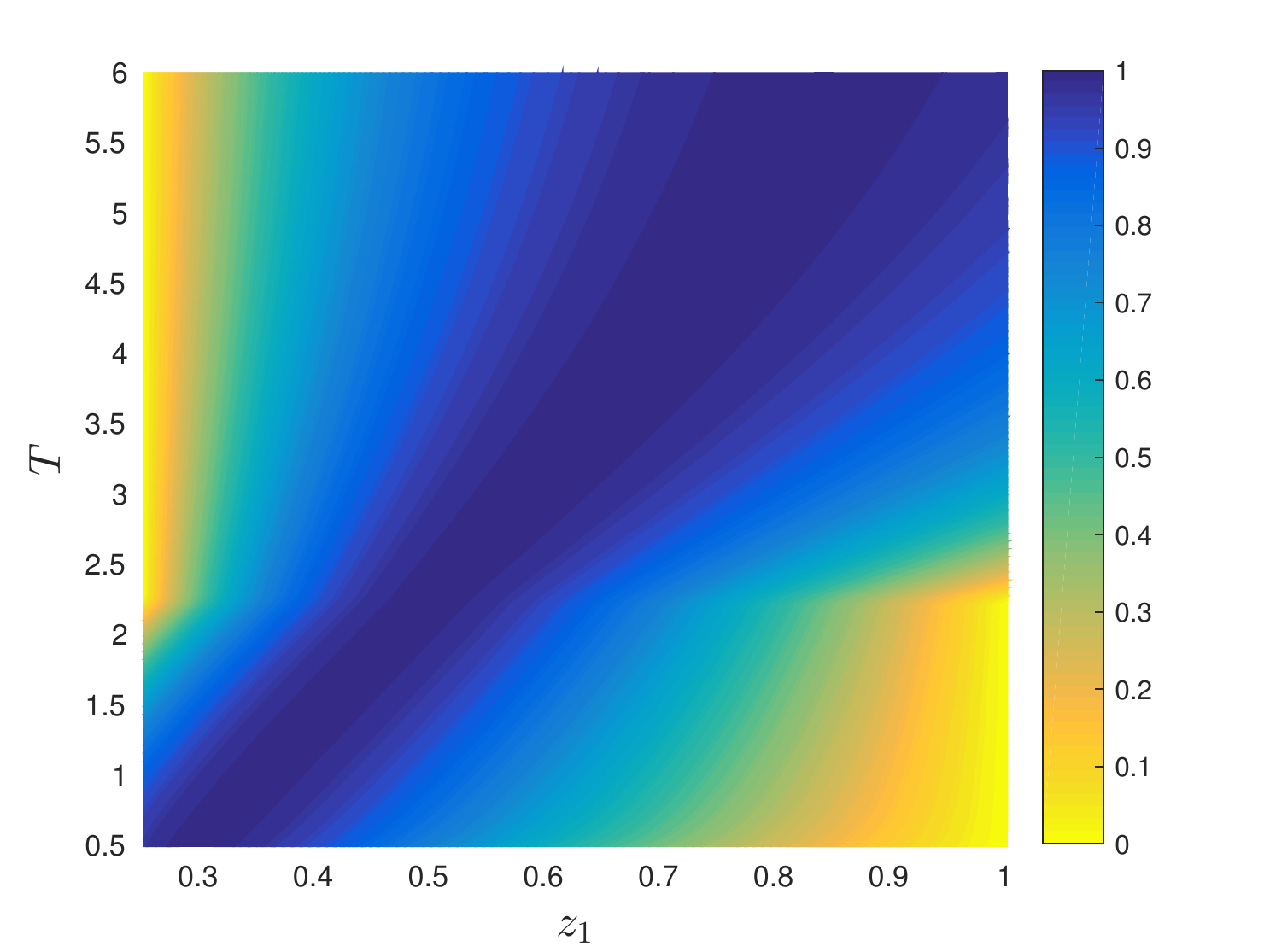}
  \caption{Normalized average reward   $J_N(T,z_1)$ as a function of $z_1 \in [z_0,1]$ and $T\in [0.5 , 6]$ for the initial condition $(x_0,s_0) = (20, 20)$.
  Contois growth function ($\mu_{max}=0.74, K_s=1$) with $s_{in}=100$, $u_{\max}=1.5$.}
  \label{fig:normalReward-z0<1}
 \end{center}
\end{figure}

\begin{figure}
 \begin{center}
 \includegraphics[width=0.8\textwidth]{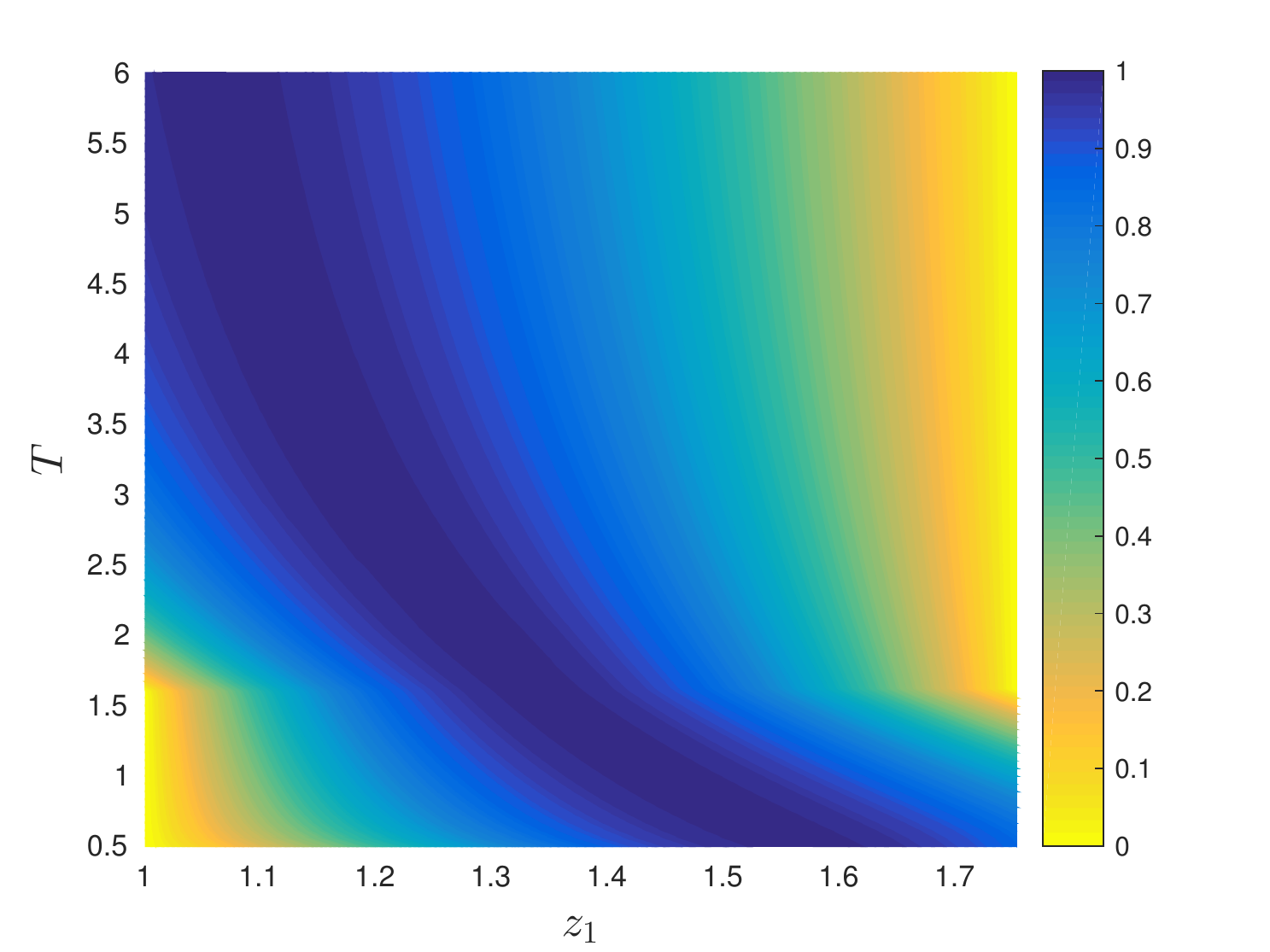}
  \caption{Normalized average reward  $J_N(T,z_1)$ as a function of $z_1 \in [1,z_0]$ and $T\in [0.5 , 6]$ for the initial condition $(x_0,s_0) = (70, 60)$.
  Contois growth function ($\mu_{max}=0.74, K_s=1$) with $s_{in}=100$, $u_{\max}=1.5$.}
  \label{fig:normalReward-z0>1}
 \end{center}
\end{figure}

This leads us to consider a new feedback that keeps the system in the set of maximizers
\begin{equation}
\overline {\cal S} = \big\{ (s,z) \in {\cal D} : s = \bar s(z) \big\}.
\end{equation}
We therefore introduce the following most rapid approach feedback to $\overline {\cal S}$
\begin{equation} \label{def:MRAPfeedback-s(z)}
\psi_{\overline {\cal S}}(s,z)=
\left|\begin{array}{ll}
  0 & \mbox{if } s>\bar s(z),\\
  \bar u(s,z) \quad &  \mbox{if } s=\bar s(z),\\
  u_{\max} & \mbox{if } s<\bar s(z),
\end{array}\right.
\end{equation}
where $\bar u(s,z)$ is the feedback that keeps the system in the set $\overline {\cal S}$, that we compute by differentiating with respect to time the equation $ s(t) = \bar s(z(t))$.

\begin{figure}
 \begin{center}
 \includegraphics[width=0.8\textwidth]{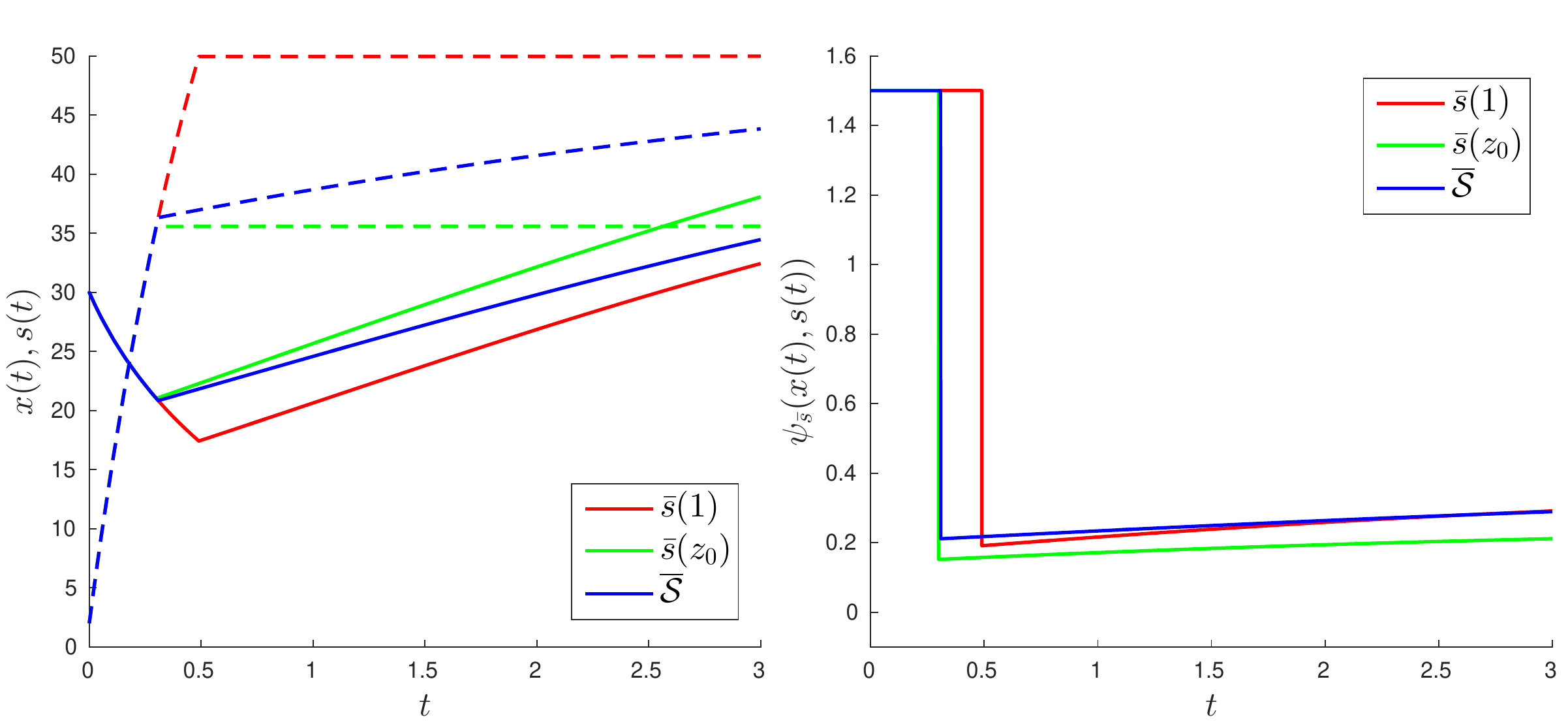}
  \caption{ On the left, $t\mapsto x(t)$ (solid lines) and $t \mapsto s(t)$ (dashed lines) with feedbacks $\psi \in \{ \psi_{\bar s(z_0)} ,  \psi_{\bar s(1)} , \psi_{\overline {\cal S}} \}$ and on the right, the corresponding open loop controls.
  Contois growth function ($\mu_{max}=0.74, K_s=1$) with $s_{in}=100$, $u_{\max}=1.5$ and initial condition $(x_0,s_0)=(30,2)$.}
  \label{fig:t->traj}
 \end{center}
\end{figure}

We first illustrate this feedback in Figure \ref{fig:t->traj} where we show the states as functions of time and the open loop realizations of the feedbacks $\psi_{\bar s(z_0)}$,  $\psi_{\bar s(1)}$ and $\psi_{\overline {\cal S}}$.
Next, in Figure \ref{fig:avg_cost_compare-T5-x020-s020-sbar(z_i)-gray} we compare the reward of the feedback $\psi_{\overline {\cal S}}$ to the others and we can notice that the reward associated with the feedback $\psi_{\overline {\cal S}}$ is always one of the best, although for any given final time it is possible to do better with a feedback $\psi_{\bar s(z_1)}$ for the right $z_1$.

Note also that the feedback $\psi_{\overline {\cal S}}$ will drive the system asymptotically towards the state $(s,z) = (\bar s ,1)$ so that it is also optimal for the infinite horizon problems \eqref{Vinfty-liminf}, \eqref{Vinfty-limsup} and \eqref{def:V0}.

\begin{figure}
 \begin{center}
 \includegraphics[width=0.8\textwidth]{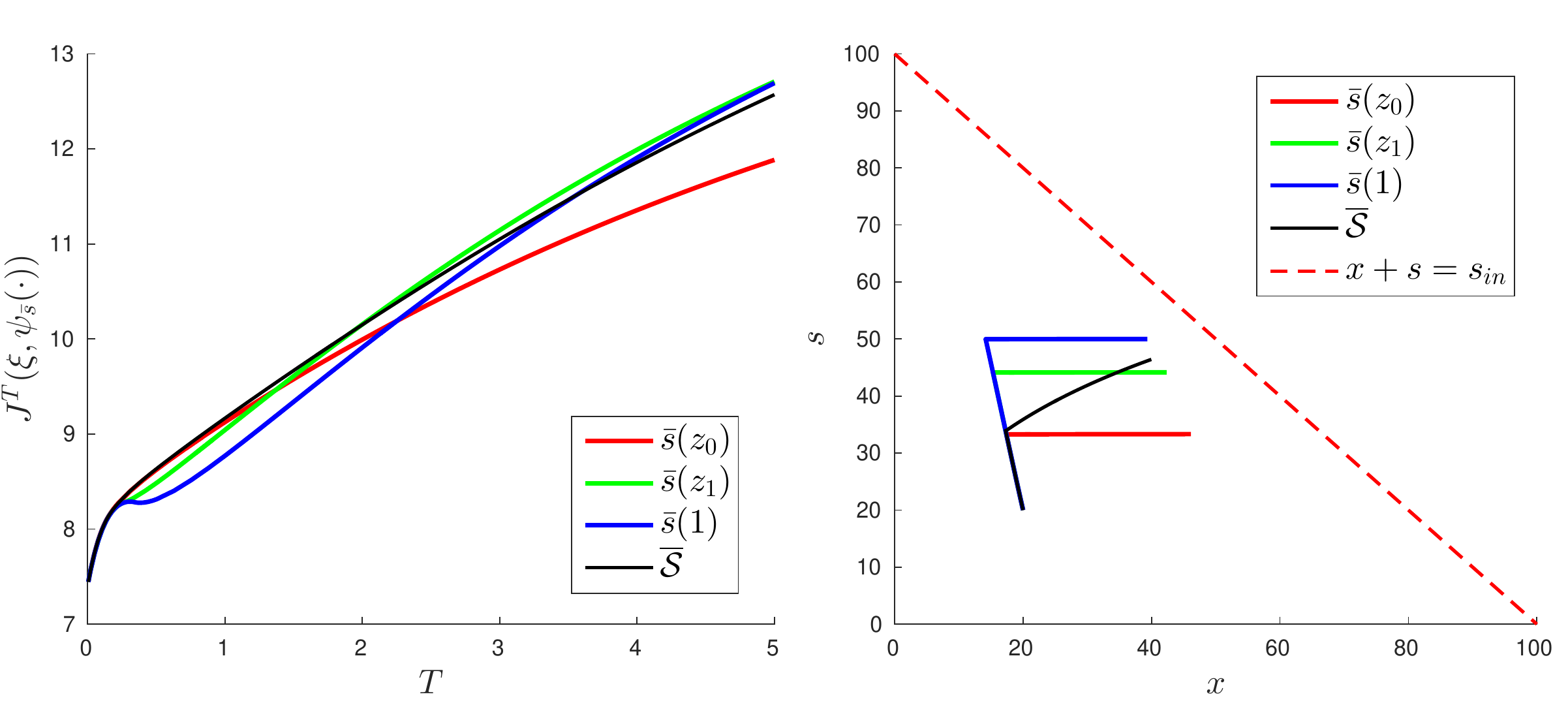}
  \caption{On the left, average reward as function of final time $T\mapsto J^T(\xi, \psi (\cdot) ) $ with feedback $\psi \in \{ \psi_{\bar s(z_0)} , \psi_{\bar s(z_1)} ,  \psi_{\bar s(1)} , \psi_{\overline {\cal S}} \}$ with $z_0 = 0.25$ and $z_1 = 0.625$.
  On the right, the corresponding state space trajectories.
  Contois growth function ($\mu_{max}=0.74, K_s=1$) with $s_{in}=100$, $u_{\max}=1.5$ and initial condition $(x_0,s_0)=(20,20)$.}
  \label{fig:avg_cost_compare-T5-x020-s020-sbar(z_i)-gray}
 \end{center}
\end{figure}

\begin{figure}
 \begin{center}
 \includegraphics[width=0.8\textwidth]{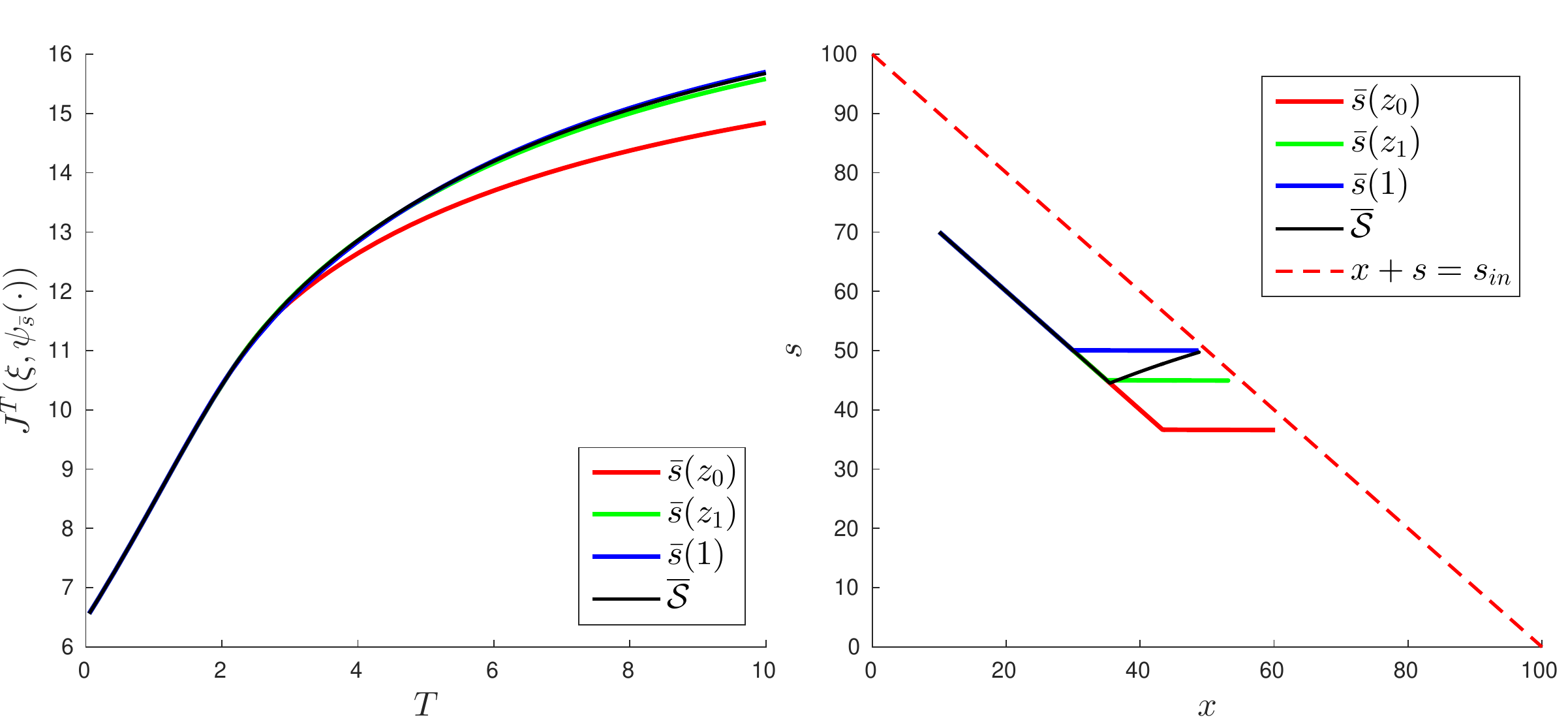}
  \caption{On the left, average reward as function of final time $T\mapsto J^T(\xi, \psi (\cdot) ) $ with feedback $\psi \in \{ \psi_{\bar s(z_0)} , \psi_{\bar s(z_1)} ,  \psi_{\bar s(1)} , \psi_{\overline {\cal S}} \}$ with $z_0 = 1/3$ and $z_1 = 2/3$.
  On the right, the corresponding state space trajectories.
  Contois growth function ($\mu_{max}=0.74, K_s=1$) with $s_{in}=100$, $u_{\max}=1.5$ and initial condition $(x_0,s_0)=(10,70)$.}
  \label{fig:avg_cost_compare-T5-x020-s020}
 \end{center}
\end{figure}

In Figure \ref{fig:cost_compare-Tf1-2-4-6-gray}, we show the difference between the rewards of the feedbacks $\psi_{\bar s(1)}$ and $\psi_{\bar s(z_0)}$ as a function of the initial condition for various final times.

From this, we see that the feedback that is best changes, depending on the initial condition and the horizon considered.

The sub-optimality estimation \eqref{eq:subOptFrame} is affected similarly, as this bound depends on the initial condition and in particular, the distance to the set $\{ z=1 \}$ has a major impact on the sub-optimality of the considered feedbacks.
In addition, the growth function has an influence on our estimation, through $W_{z_1}(\cdot)$, and we illustrate this in Figure \ref{fig:auxCostCointoisHaldaneT2} by plotting this value function for the Haldane and the Contois growth function.
Observe that, for the Contois growth function, $W_{z_1}(\cdot)$ varies significantly with the initial biomass and thus the sub-optimality bound as well.
This can be attributed to the dependence of the Contois growth function on biomass concentration and this effect is not seen with the Haldane growth function, which depends only on the substrate.

\begin{figure}
 \begin{center}
 \includegraphics[width=0.8\textwidth]{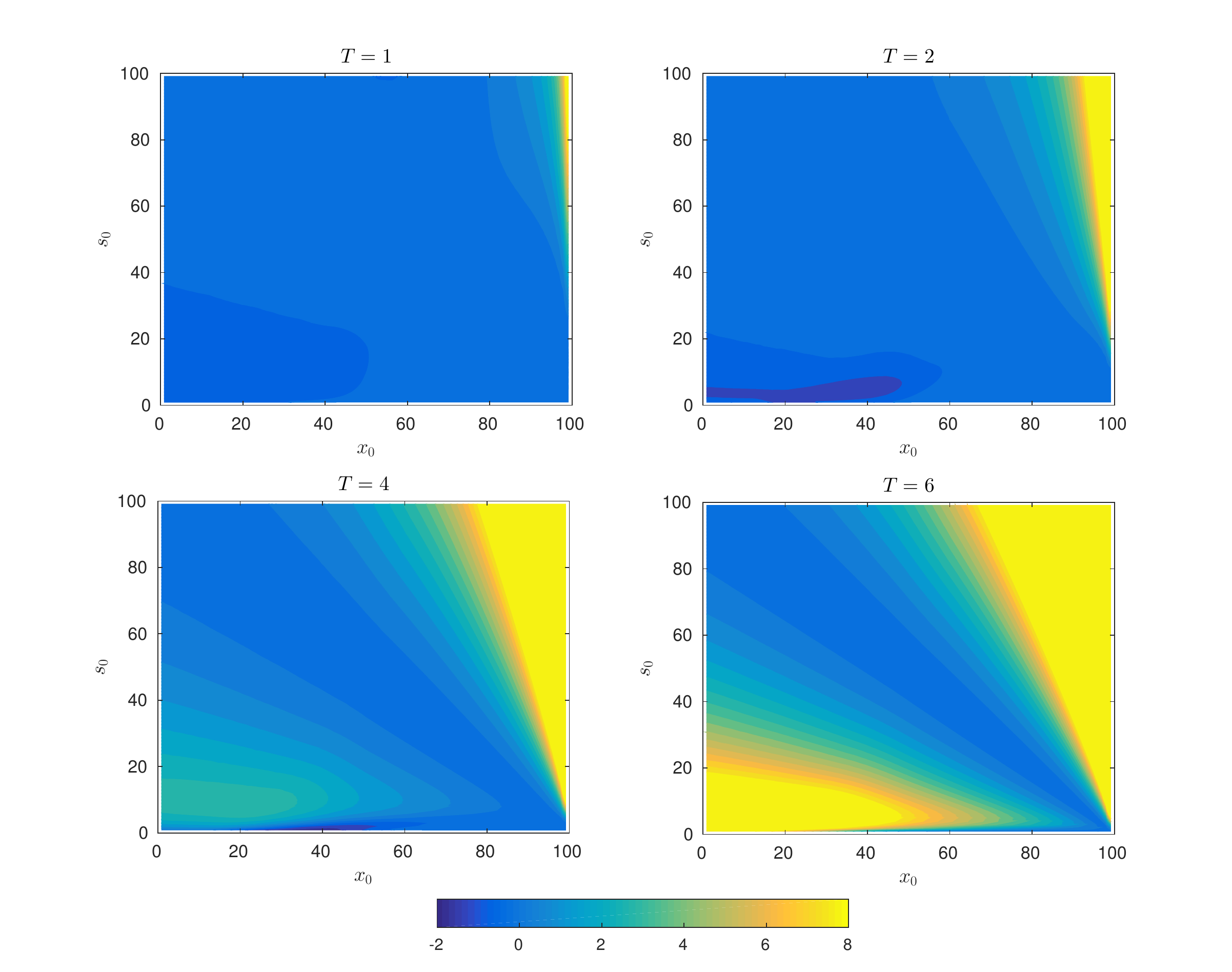}
  \caption{ Difference between rewards associated to the feedbacks $\psi_{\bar s (1)}$ and $\psi_{\bar s (z_0)}$  as functions of the initial condition  and for various final times : $ (x_0,s_0) \mapsto J(0,x_0,s_0, \psi_{\bar s (1)}(\cdot) ) - J(0,x_0,s_0, \psi_{\bar s (z_0)}(\cdot) ) $.
  Contois growth function ($\mu_{max}=0.74, K_s=1$) with $s_{in}=100$, $u_{\max}=1.5$.}
  \label{fig:cost_compare-Tf1-2-4-6-gray}
 \end{center}
\end{figure}

\begin{figure}
 \begin{center}
 \includegraphics[width=\textwidth]{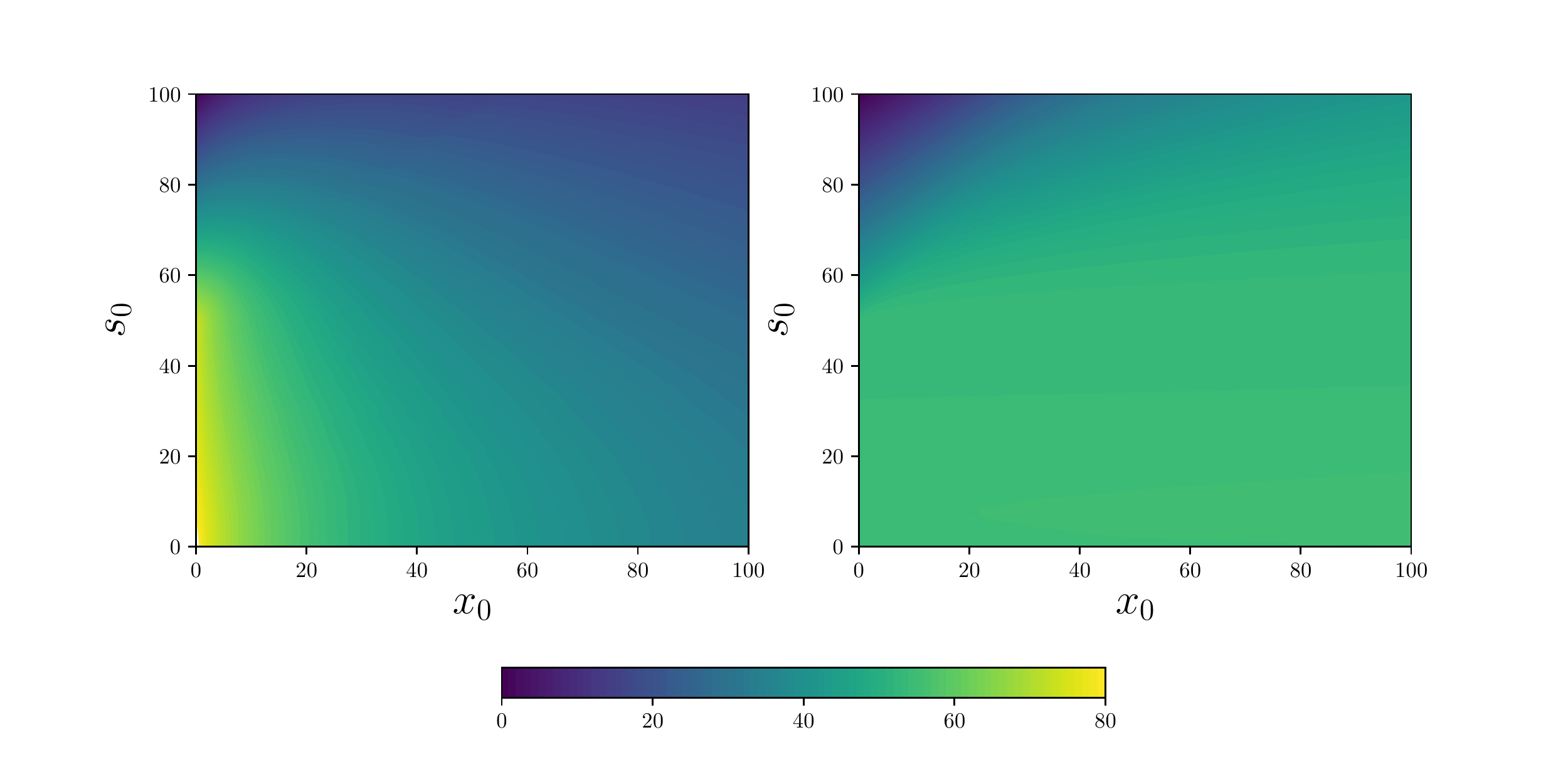}
  \caption{ Auxiliary value function $(x_0,s_0) \mapsto W_{z_1}(0,x_0,s_0)$ with $z_1 = 1$. On the left, Contois growth function ($\mu_{max}=0.74, K_s=1$, $u_{\max}=1.5$) and on the right, Haldane growth function ($\bar \mu=0.74, K_s=9.28, K_i=256$, $u_{\max}=3$). In both cases, $s_{in} = 100$ and $T=2$.}
  \label{fig:auxCostCointoisHaldaneT2}
 \end{center}
\end{figure}

We finish this section with a Lemma that shows that the considered growth functions satisfy our assumptions.

\begin{lemma} \label{lem:muAssump}
For all positive $\mu_{max}$, $\bar \mu$, $ K_s $ and $K_i $ the Monod, Haldane and Contois growth functions satisfy Assumptions \ref{assum:growthfct} and \ref{assum:maxphi-z1}.
\end{lemma}

\begin{proof}
Notice that the function $\phi$ with the Monod or Haldane function does not depend on $z$.
Let us show that the function $\mu_M$ is increasing and strictly concave
\[
\mu_M'(s) = \frac{\mu_{max}K_{s}}{(K_{s}+s)^2} >0 , \quad
\mu_M''(s) = -2\frac{\mu_{max}K_{s}}{(K_{s}+s)^3} <0 .
\]
Now, since the function $\phi(\cdot,1)$ is non-negative on $[0,s_{in}]$ and vanishes at 0 and $s_{in}$ it admits a maximum on $(0,s_{in})$.
One has
\begin{align}
\frac{d}{ds} \phi(s,1) & =  \mu_M'(s)(s_{in}-s)-\mu_M(s),\\
\frac{d^2}{ds^2}\phi(s,1) & =  \mu_M''(s)(s_{in}-s)-2\mu_M'(s).
\end{align}
The function $\phi(\cdot,1)$ is thus strictly concave on $(0,s_{in})$, which
provides the uniqueness of its maximum.

For the Haldane function, we have
\[
 \frac{d}{ds}\phi(s,1)= \bar \mu \frac{s_{in}K_s -2K_s s-s^2 (1+\frac{s_{in}}{K_i})}{(K_s +s+\frac{s^2}{K_i})^2}
\]
such that $\frac{d}{ds}\phi(0,1)>0$ and $\frac{d}{ds}\phi(s_{in},1)<0$ and since $\frac{d}{ds}\phi(\cdot,1)$ is continuous it must have an odd number of zeroes in the interval $(0,s_{in})$.
But notice that the equation $\frac{d}{ds}\phi(s,1)=0$ admits at most 2 solutions and $\phi(0,1)=\phi(s_{in},1)=0$ and therefore $\phi(\cdot,1)$ has a unique maximum.

For the Contois function, notice that $\mu_C(s,x) = \mu_M(s/x)$ so that, for $z_1 \in [\min(z_0,1) , \max(z_0,1) ]$
\[
 \phi(s,z_1) = \mu_{M}\left(\frac{s}{(s_{in}-s)z_1}\right) (s_{in}-s)
\]
and since $s \mapsto \frac{s}{(s_{in}-s)z_1}$ is an increasing function, $\phi(\cdot,z_1)$ is also strictly concave.

\end{proof}

\section{Conclusions}

In this work, we have proposed a novel approach to obtain autonomous sub-optimal feedbacks for the open problem of maximizing biogas production in the chemostat model out of equilibrium.
These controllers generalize the ``most-rapid approach path'' feedback control that is known to be optimal when the initial condition belongs to a certain manifold.
Indeed, we obtain a family of feedback controls of similar structure, for which we are able to give bounds on the sub-optimality.
This last point merits to be underlined as it usually difficult to evaluate a priori the performances of sub-optimality without having to determine or compute the optimal solution.
This choice gives also flexibility for the practitioners to choose a controller depending on the time horizon or simply to pick one when the finite horizon is poorly known (as each controller guarantees a sub-optimality bound), or to adjust it when the horizon is changed.
For infinite horizon we show that each controller guarantees the same optimal averaged cost.

This methodology, based on a framing of the dynamics, could be investigated for a larger class of dynamics, such as the two-step model, and be the matter of future work.

\section*{Acknowledgments}
The first and second authors were supported by FONDECYT grant 1160567, and by Basal Program CMM-AFB 170001 from CONICYT-Chile.
The first author was supported by a doctoral fellowship CONICYT-PFCHA/Doctorado Nacional/2017-21170249.
The third author was supported by the LabEx NUMEV incorporated into the I-Site MUSE.

%\appendix

\section*{Appendix: A Particular Example}
%\setcounter{section}{1}

%\begin{example} \label{counterex}
 We construct here a control $u(\cdot)$ for which the average rewards \eqref{averagerewardinftyinf} and \eqref{averagerewardinftysup} do not coincide.
For this, let us consider an initial condition $\xi=(s_0,z_0)=(\varepsilon, 1)$, with $\varepsilon \in (0,s_{in})$ fixed.
The set $\{ (s,1) \in \R^2_+ : s \in [0,s_{in}] \}$ is clearly invariant for the dynamics \eqref{eq:chemosz} and therefore the chosen initial condition ensures that trajectories $(s_{\xi,u}(\cdot),z_{\xi,u}(\cdot))$ remains in this set.

% In order to simplify notations, we will consider here only a substrate dependent growth function although this example also works for substrate and biomass dependent growth functions.

Now consider the 2 following paths :
\begin{enumerate}
 \item[(A)] Starting at $\xi:=(\varepsilon, 1)$, use the control $u=u_{\max}$ to reach a prescribed level of substrate $ s^*\in (\varepsilon,s_{in})$ in finite time.
   Then, apply the control $u=0$ to return to $\xi$ in finite time, which is possible by Assumption \ref{assum:umax}.
   Denote this control by $u_*$, and let $t_*$ be the (finite) time necessary to follow this path and $I_*$ be the biogas produced by this path.
 \item[(B)] Starting at $\xi:=(\varepsilon, 1)$, use $u=\mu(\varepsilon,s_{in}-\varepsilon)$ to stay at $(s=\varepsilon,z=1)$ for any time interval.
\end{enumerate}

Then, define control $u(\cdot)$ as follows:
\begin{itemize}
   \item For $t \in [0,t_*]$, set $u(t)=\mu(\varepsilon,s_{in}-\varepsilon)$ so that the biogas production for this period is $I_{\varepsilon}:= t_* \phi(\varepsilon,1)$.
   \item For $t \in (2^{2k} t_* , 2^{2k+1} t_* ]$, with $k \in \mathbb{N}$, set $u=u_*$ in order to follow the path (A) repeatedly $2^{2k}$ times. For each of these intervals the biogas production is $ 2^{2k}  I_*$.
   \item For $t \in (2^{2k+1} t_* , 2^{2k+2} t_* ]$, with $k \in \mathbb{N}$, set $u=\mu(\varepsilon,s_{in}-\varepsilon)$. For each of these intervals the biogas production is $2^{2k+1}I_{\varepsilon}$.
\end{itemize}

Thus, when we apply control $u(\cdot)$ up to a time $2^{2N} t_* $, for a given $N \geqslant 1$, the average biogas production is computed as follows
 \begin{align*}
  K_N &= \frac{1}{2^{2N} t_* } \int_0^{2^{2N} t_* } \phi(s_{\xi,u}(t),1) \, dt \\
  & = \frac{1}{2^{2N} t_* } \left(I_{\varepsilon}+ \sum_{k=0}^{N-1} 2^{2k} I_* + \sum_{k=0}^{N-1} 2^{2k+1} I_{\varepsilon}  \right)\\
%   & = \frac{1}{ t_* }\left( I_* + 2 I_{\varepsilon} \right)  \sum_{k=0}^{N-1} 2^{2(k-N)} + \frac{I_{\varepsilon}}{2^{2N} t_* } \\
  & = \frac{ I_* + 2 I_{\varepsilon}}{ t_* }  \sum_{j=1}^{N} 2^{-2j} + \frac{I_{\varepsilon}}{2^{2N} t_* }
 \end{align*}
 which yields
 $$ K_N \, \longrightarrow \, K_{\infty} := \frac{I_* + 2I_{\varepsilon}}{3 t_* }  \, \mbox{ as } N \rightarrow +\infty.$$
We have used here the fact that the sum $s_N=\sum_{j=1}^{N} 2^{-2j}$ converges to $1/3$. Indeed, this follows from the identity
 \begin{align*}
  4 s_N &=  \sum_{j=1}^{N} 2^{2(-j+1)}  = \sum_{i=0}^{N-1} 2^{-2i} = 1 + s_N - 2^{-2N}.
 \end{align*}

 However, for the same  control $u(\cdot)$,  the average biogas production is, up to time $2^{2N+1} t_* $, computed as follows
 \begin{align*}
  L_N &= \frac{1}{2^{2N+1} t_* } \int_0^{2^{2N+1} t_* } \phi(s_{\xi,u}(t),1) \, dt \\
  & = \frac{1}{2^{2N+1} t_* } \Big( 2^{2N} t_*  K_N +  2^{2N} I_* \Big) \\
  & = \frac{1}{2 } \Big(  K_N +  \frac{I_*}{ t_* } \Big)
 \end{align*}
which yields
 $$ L_N \, \longrightarrow \, L_{\infty} := \frac{2I_* + I_{\varepsilon}}{3 t_* }  \, \mbox{ as } N \rightarrow +\infty.$$

Since $ s_* > \varepsilon$, it follows that $I_* > I_{\varepsilon} $, and consequently, $L_{\infty} > K_{\infty}$.
We thus obtain that
$$\overline J^{\infty}(\xi,u(\cdot))\geqslant L_{\infty} > K_{\infty} \geqslant \underline J^{\infty}(\xi,u(\cdot)). $$

%\end{example}

\bibliographystyle{spmpsci_unsrt}
\bibliography{singlereactionbiogas}

\end{document}